\documentclass[12pt, american, reqno]{amsart} 
\usepackage{amssymb,amsthm,amsmath}
\usepackage{multirow}
\usepackage{enumerate}
\usepackage{ifthen}
\usepackage{bbm}
\usepackage{cite}
\usepackage{graphicx} 
\usepackage{tkz-euclide}
\usepackage[colorlinks,citecolor=blue,hypertexnames=false]{hyperref}

 \numberwithin{equation}{section}

\usepackage{amsfonts}
\usepackage{amscd}
\usepackage{amssymb}
\usepackage{enumerate}
\usepackage{graphicx}
\allowdisplaybreaks
\usepackage{mathabx}
\usepackage{color}
\usepackage{amsbsy}
\usepackage{graphicx}
\usepackage{amsthm}
\usepackage{amsmath}
\usepackage{amsxtra}
\usepackage{mathrsfs}
\usepackage{bbm}
\usepackage{dsfont}

\usepackage{ifthen}
\usepackage{xcolor,colortbl}
\newtheorem{theorem}{Theorem}[section]

\newtheorem{prop}[theorem]{Proposition}

\theoremstyle{definition}

\theoremstyle{remark}
\newtheorem{rem}[theorem]{Remark}
\numberwithin{equation}{section}
\definecolor{red}{rgb}{1.0, 0.0, 0.0}
\newcommand{\Bea}{\begin{eqnarray*}}
	\newcommand{\Eea}{\end{eqnarray*}}
\newcommand{\Be} {\begin{equation*}}
	\newcommand{\Ee} {\end{equation*}}
\newcommand{\be} {\begin{equation}}
	\newcommand{\ee} {\end{equation}}
\newcommand{\bea} {\begin{eqnarray}}
	\newcommand{\eea} {\end{eqnarray}}

\newcommand{\G}{\mathbb G}

\newcommand{\g}{\mathfrak{g}}

{\qed\bigskip}

\newcounter{alphabet}

\ifx\undefined\bysame
\newcommand{\bysame}{\leavevmode\hbox to3em{\hrulefill}\,}
\fi
\usetikzlibrary{patterns}
\title[ Higher order hypoelliptic damped wave equations]
{Higher order hypoelliptic damped wave equations on   graded Lie  groups with data from negative order Sobolev spaces:  the critical case}

\author{Vishvesh Kumar} 

\address{Vishvesh Kumar  \endgraf Department of Mathematics: Analysis, Logic and Discrete Mathematics	\endgraf Ghent University \endgraf Krijgslaan 281, Building S8,	B 9000 Ghent, Belgium.} \email{Vishvesh.Kumar@UGent.be and vishveshmishra@gmail.com}

\author{Shyam Swarup Mondal} \address{Shyam Swarup Mondal 
	\endgraf Stat-Math unit,
	Indian Statistical Institute Kolkata \endgraf 
	BT Road,  Baranagar, Kolkata 700108, India}
\email{mondalshyam055@gmail.com}

\author{Michael Ruzhansky} \address{Michael Ruzhansky  \endgraf Department of Mathematics: Analysis, Logic and Discrete Mathematics\endgraf Ghent University\endgraf Krijgslaan 281, Building S8,	B 9000 Ghent, Belgium\endgraf and\endgraf School of Mathematical Sciences\endgraf Queen Mary University of London, United Kingdom} \email{michael.ruzhansky@ugent.be}
\author{Berikbol T. Torebek} 

\address{Berikbol T. Torebek  \endgraf Department of Mathematics: Analysis, Logic and Discrete Mathematics	\endgraf Ghent University \endgraf Krijgslaan 281, Building S8,	B 9000 Ghent, Belgium \endgraf and\endgraf Institute of
	Mathematics and Mathematical Modeling \endgraf 28 Shevchenko str.,
	050010 Almaty, Kazakhstan.} \email{berikbol.torebek@ugent.be}

\date{\today}
\begin{document}
	\allowdisplaybreaks
	
	\begin{abstract} 	
		Let $\mathbb G$ be a graded Lie group with homogeneous dimension $Q$. In this paper, we study the Cauchy problem for a semilinear hypoelliptic damped wave equation involving a positive Rockland operator $\mathcal{R}$ of homogeneous degree $\nu\geq 2$ on $\mathbb G$  with power type nonlinearity $|u|^p$ and initial data taken from negative order homogeneous Sobolev space $\dot H^{-\gamma}(\mathbb G), \gamma>0,$ for the critical exponent case $p=1+\frac{2\nu}{Q+2\gamma}.$  We also explore the diffusion phenomenon of the higher order hypoelliptic damped wave equations on graded Lie groups with initial data belonging to Sobolev spaces of negative order. We emphasize that our results are also new, even in the setting of higher-order differential operators on $\mathbb{R}^n$,  and more generally, on stratified Lie groups.

	\end{abstract}
	\maketitle
	\tableofcontents

	\section{Introduction and main results}
	We  investigate the behavior of the solutions to the Cauchy problem  for a semilinear damped wave equation with the power  type nonlinearities of the form 
	\begin{align} \label{eq0010}
		\begin{cases}
			u_{tt}+\mathcal{R}u +u_{t} =|u|^p, & x\in  \mathbb{G},~t>0,\\
			u(0,x)=\varepsilon u_0(x),  & x\in  \mathbb{G},\\ u_t(0, x)=\varepsilon u_1(x), & x\in  \mathbb{G},
		\end{cases}
	\end{align}
	in the critical case $p=p_{\text{Crit}}(Q, \gamma, \nu):=1+\frac{2\nu}{Q+2\gamma},$ 	where $\mathcal{R}$ is a positive Rockland operator of homogeneous degree $\nu \geq 2$ on a graded Lie group $\mathbb{G}$,  and the initial data $(u_0, u_1)$ with its size parameter $\varepsilon>0$  belongs to  homogeneous Sobolev spaces of negative order $  \dot {H}^{-\gamma}(\mathbb{G}) \times  \dot {H}^{-\gamma}(\mathbb{G})$ with $\gamma>0$.  This work is a continuation of the work \cite{DKMR24} of the first three authors with Dasgupta in which we have discussed the subcritical and supercritical damped wave equation associated with Rockland operators on graded Lie groups with the initial data belonging to homogeneous Sobolev spaces of negative order.
	
	Recall that a connected and simply connected Lie group $\mathbb{G}$ is a graded Lie group if its Lie algebra $\mathfrak{g}$ is {\it graded}, that is, $\mathfrak{g}$ admits a vector space decomposition of the form $\mathfrak{g}= \bigoplus_{i=1}^\infty \mathfrak{g}_i,$ for which all but finitely many $\mathfrak{g}_i$'s are $\{0\}$ such that $[\mathfrak{g}_i, \mathfrak{g}_j] \subset \mathfrak{g}_{i+j}$ for all $i, j \in \mathbb{N}.$ We refer to Section \ref{sec2} for a detailed description of the graded Lie groups.  If the first stratum $\mathfrak{g}_1$  generates the Lie algebra $\g$ as an algebra, the group $\mathbb{G} $ is called a {stratified Lie group}. In this case, the sum of squares of a basis of vector fields in $\mathfrak{g}_1$ gives a sub-Laplacian on $\mathbb{G}$. This immediately shows that every stratified Lie group is graded. However, if the group $\mathbb{G}$ is non-stratified, then it may not have a homogeneous sub-Laplacian or Laplacian but they always possess Rockland operators. 	A Rockland operator on $\mathbb{G}$ is a left-invariant hypoelliptic differential operator of a positive homogeneous degree $\nu$, see Subsection \ref{Rocksec} for an overview.  %A left-invariant differential operator $\mathcal{R}$ on a homogeneous group $G$ called the {\it Rockland operator} if it is homogeneous of positive degree $\nu,$ that is,  $$\mathcal{R}(f \circ D_r)=r^{\nu} (\mathcal{R}f) \circ D_r, \quad r>0, f \in C^\infty(\G), $$  
	The Heisenberg group, more generally, $H$-type groups, Engel groups, and Cartan groups are examples of graded Lie groups. The following are some examples of graded Lie groups with a Rockland operator which are included in the analysis of this paper.
	
	\begin{itemize}
		\item When $\mathbb{G}=(\mathbb{R}^n,+),$ a Rockland operator $ \mathcal{R}$ can be any positive homogeneous elliptic differential operator with constant coefficients, for example, we can  consider
		\begin{align}\label{R^n}
			\mathcal{R}=(-\Delta)^m \text { or } \mathcal{R}=(-1)^m \sum_{j=1}^n a_j\left(\frac{\partial}{\partial x_j}\right)^{2 m}, \quad a_j>0, m \in \mathbb{N},
		\end{align}
		which are Rockland operators with homogeneous degree $2m$ when the commutative group $\mathbb{R}^n$ is equipped with isotropic dilations. 
		\item When  $\mathbb{G}=\mathbb{H}^n,$ the Heisenberg group, we can consider the Rockland operator of the homogeneous degree $2m$ as  
		$$ 		\mathcal{R}=(-\mathcal{L})^m \text { or } \mathcal{R}=(-1)^m \sum_{j=1}^n\left(a_j X_j^{2 m}+b_j Y_j^{2 m}\right), \quad a_j, b_j>0, m \in \mathbb{N},
		$$		where $X_j=\partial_{x_j}-\frac{y_j}{2} \partial_t,  Y_j=\partial_{y_j}+\frac{x_j}{2} \partial_t$ are the    left-invariant vector fields for its algebra $\g$ and     	$\mathcal{L}=\sum_{j=1}^n\left(X_j^2+Y_j^2\right)$  is the sub-Laplacian on    $\mathbb{H}^n$.\\
		\item When  $\mathbb{G}$ is a stratified Lie group,  then $\mathcal{L}_{\mathbb G}$, defined in  (\ref{stratified})  
		is  a positive Rockland operator with    homogeneous degree $\nu = 2$.\\
		\item When $\mathbb{G}$ is a graded Lie group with dilation weights $\nu_1, \ldots, \nu_n$,  if $\nu_0$ is any common multiple of $\nu_1, \ldots, \nu_n$, then the operators given by 
		\begin{equation} \label{rockgragd}
			\mathcal{R}:= \sum_{j=1}^n(-1)^{\frac{v_0}{v_j}} a_j X_j^{2 \frac{v_0}{v_j}}, \quad \text{with}\,\, a_1, a_2, \ldots, a_n>0,
		\end{equation}
		are  positive Rockland operators of homogeneous degree $\nu=2\nu_0$  	for any strong Malcev basis $\{X_1, X_2, \ldots, X_n\}$ of the Lie algebra $\g$.
	\end{itemize}
	
	The analysis of the semilinear damped wave equation is related to the analysis of the semilinear heat equation due to its diffusive nature \cite{BF89,RG98, HKN04}.  The semilinear heat equation on the Heisenberg group was first time studied by Zhang \cite{Zhang1} extending the seminal works of Fujita \cite{Fujita66}. He proved that the Fujita exponent for the semilinear heat equation of the Heisenberg group is $$p_{\text{Fuji}}(Q):=1+\frac{2}{Q},$$ where $Q$ is the homogeneous dimension of the Heisenberg group. In \cite{Pas98} Pascucci extended the above result to stratified Lie groups. We refer to \cite{Yang, KTT20}  for the study of the Fujita exponent for heat equations associated with Rockland operators on graded Lie groups, and to \cite{RY} for the Fujita exponent of the semilinear heat equation related to the sub-Laplacian on general unimodular Lie groups and sub-Riemannian manifolds.     In their work, Georgiev and Palmieri \cite{Vla} explored the global existence and nonexistence results for the Cauchy problem associated with the semilinear damped wave equation with $L^1$-initial data on the Heisenberg group, focusing on nonlinearities of the power type, such as $|u|^p$. They identified the critical exponent as the Fujita exponent $p_{\text{Fuji}}(Q):=1+\frac{2}{Q}$, which serves as the threshold for determining whether global-in-time Sobolev solutions exist with small data or whether local-in-time weak solutions will blow up. The analogous critical exponent for the semilinear damped wave equation with power-type nonlinearities in the Euclidean setting has been investigated in works such as \cite{IKeta and Tanizawa, Matsumura, Zhang, Todorova}, among others. Furthermore, Palmieri \cite{Palmieri2020} derived $L^2$-decay estimates for solutions to the homogeneous linear damped wave equation on the Heisenberg group, as well as for their time derivatives and horizontal gradients. In \cite{30}, the authors studied the Cauchy problem for the semilinear damped wave equation involving the Rockland operator on the graded Lie groups with power-like nonlinearities and established the global-in-time well-posedness for small data in the presence of positive mass and damping terms. 
	
	In recent years, considerable attention has been devoted by several researchers to finding new critical exponents for the semilinear damped wave equations in different frameworks. For instance, in the Euclidean setting, for the semilinear damped wave equations with initial data belonging to $L^m$-space with  $m \in(1,2]$, the modified critical exponent becomes $$p_{\text {Crit}}(n):=p_{\text {Crit }}\left(\frac{n}{m}\right)=1+\frac{2 m}{n}.$$ For complete detail of global (in-time) existence for small data solutions and blow-up solutions, we refer to  \cite{Ikeda2002,Ikeda2019,Nakao93} and reference therein. In contrast to the $L^1$-case \cite{Zhang}, it was observed that for $L^m, m\in (1, 2]$ data, the critical exponent $p_{\text {Crit}}(n):=p_{\text {Crit }}\left(\frac{n}{m}\right)=1+\frac{2 m}{n},$ belongs to the small data global existence case. Since the Hardy-Littlewood-Sobolev inequality implies that $L^m(\mathbb{R}^n)\hookrightarrow \dot{H}^{-\gamma}(\mathbb{R}^n)$ for $\gamma:= n \left( \frac{1}{m}-\frac{1}{2} \right) \in [0, \frac{n}{2}),$ several researchers began the analysis of semilinear evolution equations with initial data from Sobolev spaces of negative orders, we refer to \cite{GW, Reissig, Dcri, DKMR, TZZ24, DKMR24} and references therein. Specifically, Chen and Reissig \cite{Reissig} studied the semilinear damped wave equation (\ref{eq0010}) with $\G=\mathbb{R}^n$ and $\mathcal{R}=- \Delta_{\mathbb{R}^n}$ with initial data additionally belonging to homogeneous Sobolev spaces $\dot H^{-\gamma}(\mathbb{R}^n)$ of negative order  $-\gamma$. In this case, they found a new critical exponent given by $$p_{\text{Crit}}(n, \gamma):= 1+ \frac{4}{n+2\gamma}$$ for some $\gamma\in (0, \frac{n}{2}).$ This exponent can be seen as a generalization of the second critical exponent in the sense of Lee and Ni \cite{LN} derived considering $L^m$-initial data with $m \in (1, 2].$ The behavior of the solution at the critical exponent  $p_{\text{Crit}}(n, \gamma):= 1+ \frac{4}{n+2\gamma}$ was recently examined in \cite{Dcri}.  The analysis of damped wave equations on the Heisenberg group in the framework of negative order Sobolev space has been carried out in \cite{DKMR, Dcri}.

	In \cite{DKMR24}, the first three authors with Dasgupta initiated the analysis of the semilinear hypoelliptic damped wave equation (\ref{eq0010})  associated with a positive Rockland operator $\mathcal{R}$  on a graded Lie group $\G$ of homogeneous dimension $Q \leq 6$   with power type nonlinearity $|u|^p$ and initial data taken from negative order homogeneous Sobolev space $\dot H^{-\gamma}(\mathbb G), \gamma>0$. We found that the exponent 
	$$p=p_{\text {Crit }}(Q, \gamma, \nu):=1+\frac{2\nu}{Q+2\gamma}$$
	is the new critical exponent of the damped wave equation \eqref{eq0010} in the sense that 
	\begin{itemize}
		\item    global-in-time existence of small data Sobolev solutions of lower regularity for $p>p_{\text {Crit }}(Q, \gamma, \nu) $ in some energy evolution space; and
		\item blow-up of weak solutions in finite time even for small data for $p<p_{\text {Crit }}(Q, \gamma, \nu)$.
	\end{itemize}
	
	However,  the question of either proving the global (in time) existence of small data Sobolev solutions or the blow-up of weak solutions in the critical case $p:= p_{\text {Crit }}(Q, \gamma, \nu)$ was still {\it open}.

	The following result is one of the main findings of this paper, resolving the open question of \cite{DKMR24} by proving the small data global existence for the critical case.
	\begin{theorem}\label{main-result1}  Let ${\mathbb{G}}$ be a graded Lie group of homogeneous dimension $Q$ and let $\mathcal{R}$ be a positive Rockland operator of homogeneous degree $\nu \geq 2.$ Assume that 
		$$ \begin{cases}
			\gamma \in\left(0, \frac{Q}{2}\right) 
			\quad  \hspace{3.4cm} \text{if}  \quad Q=1, 2;\\
			\gamma \in\left(0, \min\{\frac{Q}{2}, \tilde{\gamma} \} \right)
			\quad  \hspace{2cm} \text{if}  \quad Q=3, 4;\\
			\gamma \in \left( \frac{Q\nu}{2s}-\frac{Q}{2} -\nu, \min\{\frac{Q}{2}, \tilde{\gamma} \}\right)
			\quad  \hspace{0.2cm} \text{if}  \quad Q=5, 6,
		\end{cases}$$ where $\tilde{\gamma}$ denotes the positive root of  the quadratic equation $2 \tilde{\gamma}^2+Q \tilde{\gamma}-\nu Q=0$, i.e., $\tilde{\gamma}= \frac{-Q+\sqrt {Q^2+8\nu Q}}{4}$ for $Q\geq 3$.   Also,  let the exponent $p$ satisfy
		\begin{align}
			p=p_{\text {Crit }}(Q, \gamma, \nu):=1+\frac{2\nu}{Q+2\gamma}.  		\end{align}
		Then, there exists a small positive constant $\varepsilon_0$ such that for any $\left(u_0, u_1\right) \in \mathcal{A}^{s }:=\ (H^s \cap \dot{H}^{-\gamma}\ ) \times\ (L^2 \cap \dot{H}^{-\gamma}\ )$ satisfying $\left\|\left(u_0, u_1\right)\right\|_{\mathcal{A}^{s}}=\varepsilon \in\left(0, \varepsilon_0\right]$, the Cauchy problem for the semilinear damped wave equation (\ref{eq0010}) has a uniquely determined Sobolev solution
		$$
		u \in \mathcal{C}\left([0, \infty), H^s\right).
		$$

	\end{theorem}

	\begin{table}[t] \label{tabe}
		\setlength{\tabcolsep}{12pt}
		\renewcommand{\arraystretch}{2}
		\begin{tabular}{|c|c|c|c|c|}
			\hline
			\cellcolor{gray!25} $Q$ 
			& \cellcolor{gray!25} $\nu$ 
			& \cellcolor{gray!50} Global Existence 
			&\cellcolor{gray!65}  Blow-up\\
			\hline
			1, 2   &  $\geq 2$ &     $ 
			1+\frac{2\nu}{Q+2 \gamma} \leq p    
			\leq   \frac{Q}{(Q-2s)_+} 
			$    &  $1<p < 1+\frac{2\nu}{Q+2 \gamma}$ \\ 
			\hline
			3 &  2 &  $ 
			\begin{array}{l}
				1+\frac{4}{3+ 2\gamma} < p \leq  \frac{Q}{Q- 2s} ~\quad \text { if } ~0<\gamma \leq \tilde{\gamma}  \\p=1+\frac{4}{3+ 2\gamma}  ~\quad \text { if } ~0<\gamma < \min\{\frac{Q}{2}, \tilde{\gamma} \}  \\
				1+\frac{ 2\gamma}{Q}  \leq   p \leq \frac{Q}{Q- 2s} ~\quad\quad \text { if }  ~\tilde{\gamma} <\gamma<\frac{Q}{2} 
			\end{array}
			$  &  $  1<p<1+\frac{4}{3+2 \gamma}$   \\ 
			\hline
			3    &  4&  $  \begin{array}{l}
				1+\frac{8}{3+2 \gamma} < p    
				\leq   \frac{Q}{(Q-2s)} \\ p=1+\frac{8}{3+2 \gamma} ~\quad \text { if } ~0<\gamma < \min\{\frac{Q}{2}, \tilde{\gamma} \}  \end{array}
			$    &    $1<p < 1+\frac{8}{3+2 \gamma}$   \\
			\hline
			4, 5, 6&   2  &  $ 
			\begin{array}{l} \vspace{0.3cm}
				1+\frac{2\nu}{Q+ 2\gamma} < p \leq  \frac{Q}{Q- 2s} ~\quad \text { if } ~0<\gamma \leq \tilde{\gamma}  \\ p=1+\frac{2\nu}{Q+ 2\gamma}  ~\quad \text { if } ~ \frac{Q\nu}{2s}-\frac{Q}{2} -\nu<\gamma < \min\{\frac{Q}{2}, \tilde{\gamma} \}  \\
				1+\frac{ 2\gamma}{Q} \leq  p \leq \frac{Q}{Q- 2s} ~\quad\quad \text { if }  ~\tilde{\gamma} <\gamma<\frac{Q}{2} 
			\end{array}
			$  &  $1<p<1+\frac{2\nu}{Q+2 \gamma}$ \\
			\hline
		\end{tabular}
		\vspace{10pt}
		\caption{Ranges of  $p$ for global-in-time
			existence and blow-up of weak solutions for a pair $(Q, \nu)$ with the Rockland operator defined in \eqref{rockgragd}}
		\vspace{-15pt}
		\label{Table3}
	\end{table}

	The critical exponent $p_{\text{Crit}}(Q, \gamma, \nu)$  is new even in the setting of higher-order homogeneous differential operators (such as powers of negative Laplacian) (\ref{R^n})
	on $\mathbb{R}^n$, and, more generally, for a negative sublaplacian and its powers
	on a stratified Lie group  $\mathbb{G}$. 
	As an application of the main results of \cite{DKMR24} and Theorem \ref{main-result1} for the special choice of the Rockland operator of homogeneous degree $\nu=2 v_0$ \begin{equation} 
		\mathcal{R}:= \sum_{j=1}^n(-1)^{\frac{v_0}{v_j}} a_j X_j^{2 \frac{v_0}{v_j}}, \quad \text{with}\,\, a_1, a_2, \ldots, a_n>0,
	\end{equation}
	as defined in \eqref{rockgragd}, we have recorded in Table \ref{tabe} the precise descriptions of the qualitative behavior of the solutions to higher order hypoelliptic damped wave equations on graded Lie groups.
	
	To prove this theorem, we first established the following linear estimates for $L^m, m \in (1, 2]$ on graded Lie groups, which extend the linear estimates of the damped wave equation for $L^1$-data in the setting of the Heisenberg group \cite{Palmieri2020}. This result is of independent interest and will be extremely useful to study the nonlinear damped wave equations with $L^m$-data.
	
	\begin{prop}  Let ${\mathbb{G}}$ be a graded Lie group of homogeneous dimension $Q$ and let $\mathcal{R}$ be a positive Rockland operator of homogeneous degree $\nu \geq 2.$
		Let $s \in (0, 1]$ and  $(u_0, u_1) \in \left( H^s \cap L^m   \right) \times \left(L^2   \cap L^m  \right) $ for some $m \in(1,2]$.
		Then the  solution  of  the   linear Cauchy problem   	\begin{align}\label{Linear-system}
			\begin{cases}
				u_{tt}+\mathcal{R}u +u_{t} =0, & x\in  \mathbb{G},~t>0,\\
				u(0,x)= u_0(x),  & x\in  \mathbb{G},\\ u_t(0, x)= u_1(x), & x\in  \mathbb{G},
			\end{cases}
		\end{align}
		satisfies the following decay estimate	
		$$
		\begin{aligned}
			&\|u(t, \cdot)\|_{L^2  } \leq C(1+t)^{-\frac{Q}{\nu}\left(\frac{1}{m}-\frac{1}{2}\right)}  \left(\left\|u_0\right\|_{L^2   \cap L^m  }+\left\|u_1\right\|_{H ^{-1}   \cap L^m  }\right), \\
			& \left\|\mathcal{R}^{\frac{s}{\nu}} u(t, \cdot)\right\|_{L^2  }  \leq C1+t)^{-\frac{Q}{\nu}\left(\frac{1}{m}-\frac{1}{2}\right)-\frac{s}{\nu}}\left(\left\|u_0\right\|_{H^s \cap L^m}+\left\|u_1\right\|_{H^{s-1}    \cap L^m  }\right) ,
		\end{aligned}
		$$
		for $m \in (1, 2].$
	\end{prop}

	Now  consider the following Cauchy problem for the heat equation  
	\begin{align} \label{heat}
		\begin{cases}
			&w_t+\mathcal{R}w=0, \quad  g\in \G, t>0, \\ &w(0, x)=u_0(x)+u_1(x), \quad g\in \G,
		\end{cases}
	\end{align}
	where the initial data $u_0, u_1$ are the same as in (\ref{eq0010}). One of the natural and interesting questions is whether it is possible to provide an asymptotic profile of the solution to (\ref{eq0010}) given by a solution of (\ref{heat}) as time tends to infinity. Interestingly, we show how the diffusion phenomenon bridges decay properties of solutions to the Cauchy problem for the damped wave equation (\ref{eq0010}) and solutions to the   Cauchy problem for the heat equation  (\ref{heat}). By diffusion phenomenon, we mean that when measuring the difference between Sobolev solutions of the damped wave equation and the heat equation in appropriate norms, an additional time decay rate emerges. Such diffusion phenomenon with the initial data from $L^m$-space or the negative order Sobolev spaces in the Euclidean setting have been observed in the literature previously, we refer \cite{Reissig, HL92, Nis03,Kar00} and references therein.  
	
	The following theorem is the second main
	result of this paper, addressing that the diffusion phenomenon is also valid in the framework of the negative order Sobolev space $\dot{H}^{-\gamma}$ on a graded Lie group $\G$. This result is also new for the Heisenberg group.

	\begin{theorem} Let ${\mathbb{G}}$ be a graded Lie group of homogeneous dimension $Q$ and let $\mathcal{R}$ be a positive Rockland operator of homogeneous degree $\nu \geq 2.$
		Let	$\left(u_0, u_1\right) \in\left(
		{	H} ^{s}  \cap \dot 
		{	H} ^{-\gamma} \right) \times\left(
		{	H} ^{s}   \cap  \dot
		{	H} ^{-\gamma}  \right)$  with  $s \geq 0$ and $\gamma \in \mathbb{R}$ such that   $s+\gamma+\nu \geqslant 0$. Let $u$ and $w$ be the solutions to the linear Cauchy problems (\ref{Linaer equation}) and  (\ref{heat}), respectively. Then, $u-w$ satisfies
		\begin{align}
			\|u(t, \cdot)-w(t, \cdot)\|_{
				\dot	{	H} ^{s} } \lesssim(1+t)^{-\frac{s+\gamma}{\nu}-1}\left(\left\|u_0\right\|_{
				{	H}^{s}  \cap \dot
				{	H} ^{-\gamma} }+\left\|u_1\right\|_{
				{	H}^{s-1}  \cap \dot
				{	H} ^{-\gamma}  }\right) .
		\end{align}
	\end{theorem}

	Apart from the introduction, the outline of the organization of the paper is as follows.
	\begin{itemize}
		\item  Section \ref{sec2} is devoted to recalling some basics of the Fourier analysis on graded Lie groups to make the paper self-contained. 
		\item  In Section \ref{sec3}, we discussed the philosophy of our approach to established the global existence result and defined the notion of suitable solution and related energy spaces. 
		\item  In Section \ref{sec4}, we prove the small data global existence results for the higher order hypoelliptic damped wave equation on graded Lie groups. 
		\item In Section \ref{sec5}, we explore the diffusion phenomenon of the higher order hypoelliptic damped wave equation on graded Lie groups with initial data belonging additionally
		to Sobolev spaces of negative order.  
	\end{itemize}

	\section{Preliminaries: Analysis on    graded Lie  groups} \label{sec2}
	For more details on the material of this section, we refer to \cite{Folland, Fischer, RF17}.
	\subsection{Graded Lie groups} A graded Lie group $\mathbb{G}$  is a connected and simply connected  Lie group whose Lie algebra $\mathfrak{g}$ is {\it graded}, that is, $\mathfrak{g}$ admits a vector space decomposition  $\mathfrak{g}= \bigoplus_{i=1}^\infty \mathfrak{g}_i,$ for which all but finitely many $\mathfrak{g}_i$'s are $\{0\}$ and satisfy the inclusions $[\mathfrak{g}_i, \mathfrak{g}_j] \subset \mathfrak{g}_{i+j}$ for all $i, j \in \mathbb{N}.$ Such a decomposition of a Lie algebra $\mathfrak{g}$ is called {\it gradation} of $\mathfrak{g}.$ A graded Lie algebra is {\it stratifiable} if there exists a gradation of $\mathfrak{g}$ such that $[\mathfrak{g}_1, \mathfrak{g}_i ]= \mathfrak{g}_{i+1}$ for all $i \in \mathbb{N}.$  A Lie group $\mathbb{G}$ whose Lie algebra $\g$ is stratifiable is called a {stratified Lie group}.  This immediately shows that every stratified Lie group is graded.  The Heisenberg group, more generally, $H$-type groups, Engel groups and Cartan groups are examples of stratified Lie groups.

	We define a family of dilations $D_r^{\mathfrak{g}},\,r>0,$ on a Lie algebra $\mathfrak{g}\cong \mathbb{R}^n $ as the vector space automorphisms of $\mathfrak{g}$ of the form $D_r^{\mathfrak{g}}:=\exp(\ln(r)A)$ for some diagonalisable matrix $A \sim \text{diag}[\nu_1, \nu_2, \ldots, \nu_n]$ with positive eigenvalues $0<\nu_1 \leq  \nu_2\leq \ldots \leq \nu_n$ on $\mathfrak{g}$ such that   $$D_r^{\mathfrak{g}}[X, Y]=[D_r^{\mathfrak{g}}X, D_r^{\mathfrak{g}}Y],$$
	for all $X, Y \in \mathfrak{g}$ and $r>0.$ The  positive eigenvalues $0<\nu_1 \leq  \nu_2\leq \ldots \leq \nu_n$ of $A$ are called the dilations' weights of $\g.$  A Lie algebra $\g \cong \mathbb{R}^n$ is called {\it homogeneous} if there exists a family of dilations $D_r^{\mathfrak{g}},\,r>0,$ on $\g.$ It is well known that the existence of a family of dilations on $\g$ implies that $\g$ is a nilpotent Lie algebra. The Lie group $\mathbb{G}:=\exp{\g},$ which is connected and simply connected, is called  {\it homogeneous} if $\g$ is a homogeneous Lie algebra.  The family of dilations $\{D_r^{\mathfrak{g}}: r>0\}$ on Lie algebra $\mathfrak{g}$ induces a family of dilations $\{D_r: r>0\}$ on the group $\mathbb{G}$ by $D_r:=\exp \circ D_r^{\mathfrak{g}} \circ \exp^{-1},\,\, r>0.$

	It is easy to note that for a given  graded Lie algebra $\mathfrak{g}= \bigoplus_{i=1}^\infty \mathfrak{g}_i$, the sequence of subspaces $\mathfrak{I}_k:=\bigoplus_{i=k}^\infty \mathfrak{g}_i$ forms a finite nested sequence of ideals in $\g.$ Thus, using these ideals $\mathfrak{I}_k,$ any basis $\{X_1, X_2, \ldots, X_n\}$ given as the union of the bases $\{X_1, X_2, \ldots, X_{n_i}\}$ is necessarily a strong Malcev basis of $\g.$ Such a basis of a graded Lie algebra $\g$ gives rise to a  family of dilations $D_r^{\g}, \,r>0,$ on $\g$ using the matrix given by  $A X_j=i X_j$ for every $X_j \in \g_i,$ that is, $D_r^{\g} X_j=r^{i}X_j.$ 
	
	We may identify $\mathbb{G}$ with $\mathbb{R}^{n}$ with $n=\text{dim}~\mathfrak{g}$ via the exponential map $\exp: \mathfrak{g} \rightarrow G$ given by $x=\exp(x_1X_1+x_2X_2+\cdots+x_nX_n) \in \mathbb{G}$ having a basis of $\g.$ Using this identification, we can naturally identify a function on $\mathbb{G}$ as a function on $\mathbb{R}^n.$ These exponential coordinates allow us to represent the action of $D_r$ on $\mathbb{G}$ explicitly:
	$$D_r(x)=rx:=\exp(r^{\nu_1}X_1+r^{\nu_2}X_2+\cdots+r^{\nu_n}X_n)=(r^{\nu_1}x_1, r^{\nu_2}x_2, \ldots, r^{\nu_n}x_n),$$
	for $x:=(x_1,x_2,\ldots,x_n) \in \mathbb{G},\,\,r>0.$ This notion of dilation on $\mathbb{G}$ is crucial to define the notion of homogeneity for functions, measures, and operators. For examples, the bi-invariant Haar measure $dx$ on $\mathbb{G},$ which is just a Lebesgue measure on $\mathbb{R}^n,$ is $Q$-homogeneous in the sense that 
	$$d(D_r(x))=r^Q dx,$$ where $Q:= \sum_{i=1}^\infty i \dim
	\g_i=\nu_1+\nu_2+\ldots+\nu_n$ is called the homogeneous dimension of $\mathbb{G}.$ It is customary to jointly rescale weights so that $\nu_1=1.$ This also shows that $Q \geq n.$
	Identifying the elements of $\g$ with the left-invariant vector fields, each $X_j$ is a homogeneous differential operator of degree $\nu_j.$ For every multi-index $\alpha \in \mathbb{N}^n_0$, we set $X^\alpha= X_1^{\alpha_1} X_2^{\alpha_2}\ldots X_r^{\alpha_n}$ in the universal enveloping algebra $\mathfrak{U}(\mathfrak{g})$ of the Lie algebra $\mathfrak{g}.$ Then $X^\alpha$ is of homogeneous degree $[\alpha]:=\alpha_1 \nu_1+\alpha_2 \nu_2+\ldots+\alpha_n \nu_n.$
	
	\subsection{Homogeneous quasi-norms on homogeneous  groups and polar decomposition} 
	A {\it homogeneous quasi-norm} on a homogeneous group $\G$ is a continuous function $|\cdot|:\G \rightarrow [0, \infty)$ such that it satisfies the following properties:
	\begin{itemize}
		\item $|x|=0$ if and only if $x=e_\G.$
		\item  $|x^{-1}|=|x|$
		\item $|D_rx|=r|x|$ for $r>0.$
	\end{itemize}
	There always exists a homogeneous quasi-norm in any homogeneous group $\mathbb{G}$ (\cite{Folland}). One can show the existence of a homogeneous quasi-norm on $\mathbb{G}$, which is $C^\infty$-smooth on $\mathbb{G} \backslash \{e_\mathbb{G}\}.$ Every homogeneous quasi-norm satisfies the following triangle inequality with the constant $C\geq 1:$
	$$|xy| \leq C (|x|+|y|)\quad \forall x, y \in \mathbb{G}.$$
	In fact, it is always possible to choose a  homogeneous quasi-norm on any homogeneous group that satisfies the triangle inequality with constant $C=1.$ Any two homogeneous quasi-norms on $\mathbb{G}$ are equivalent.
	
	Similar to the Euclidean space, there is a notion of polar decomposition on a homogeneous group $\mathbb{G}$ for the homogenous quasi-norm $|\cdot|.$  Let 
	$$\mathfrak{S}:=\{x \in G:\,\,\, |x|=1\}$$ be the unite sphere with respect to the homogenous quasi-norm $|\cdot|.$ Then there exists a unique Radon measure $\sigma$ on $\mathfrak{S}$ such that for all $f \in L^1(\G),$ we have 
	\begin{equation} \label{polardeco}
		\int_{G} f(x)dx= \int_0^\infty \int_{\mathfrak{S}} f(ry) r^{Q-1} d\sigma(y)\, dr.
	\end{equation}
	
	\subsection{Positive Rockland operators on graded Lie groups} \label{Rocksec}
	
	Now it is time to introduce the main object of the discussion, namely, Rockland operators. To define them, we first need to fix some notation for the continuous unitary representations of the group. Let $(\pi, \mathcal{H}_\pi)$ be a continuous unitary representation of a graded Lie group $\G.$ Denote the set of equivalence classes of all strongly continuous unitary representations of $\G$ by $\widehat{\G}.$ Here, the Hilbert space $\mathcal{H}_\pi$ denotes the representation space of $\pi.$ We also denote the space of all smooth vectors of $\pi$ by $\mathcal{H}_\pi^\infty,$ which is a subspace of $\mathcal{H}_\pi.$
	The infinitesimal representation of the Lie algebra $\mathfrak{g}$ and its extension to the universal enveloping Lie algebra $\mathfrak{U}(\mathfrak{g})$ will also be denoted by $\pi.$ We note here that the space of left-invariant vector fields and the algebra of left-invariant differential operators on $\G$ can be identified with $\mathfrak{g}$ and $\mathfrak{U}(\mathfrak{g}),$ respectively. For a left-invariant differential operator $T$, let us denote by $\pi(T),$ the infinitesimal representation $d\pi(T)$ associated with $\pi \in \widehat{\G}.$ 
	
	A left-invariant differential operator $\mathcal{R}$ on a homogeneous group $\G$ is called a {\it Rockland operator} if it is homogeneous of positive degree $\nu,$ that is, 
	$$\mathcal{R}(f \circ D_r)=r^{\nu} (\mathcal{R}f) \circ D_r, \quad r>0,\,\, f \in C^\infty(\G), $$ and 
	the operator $\pi(\mathcal{R})$ is injective on $\mathcal{H}^\infty_\pi$ for every nontrivial representation $\pi \in \widehat{\G},$ that is, 
	\begin{equation}\label{Rock}
		\forall v \in \mathcal{H}^\infty_\pi \quad \pi(\mathcal{R})v=0 \implies v=0.
	\end{equation}
	The condition \eqref{Rock} is known as the Rockland condition. The Rockland condition for $\mathcal{R}$ is equivalent to the {\it hypoellipticity} of $\mathcal{R},$ that is locally, $\mathcal{R}f \in C^\infty(\G) \implies f \in C^\infty(\G).$ This equivalence is commonly known as the {\it Rockland conjecture} (see \cite{Rock}) and was resolved in \cite{HN} (see also \cite{Miller}).
	A Rockland operator is {\it positive} when 
	$$\int_{\G} \mathcal{R}f(x) \overline{f(x)} dx \geq 0,\quad \forall f \in \mathcal{S}(\G).$$ 
	
	It is a celebrated result of Miller \cite{Miller}  (see also \cite{terRob97} and  \cite[Proposition 4.1.3]{Fischer}) which says that if there is a Rockland operator on a homogeneous Lie group $\G$, then the group $\G$ must be graded.
	On the other hand, an infinite family of positive Rockland operators can be created for any graded Lie group $\G.$ Indeed, the operators given by 
	\begin{equation}
		\mathcal{R}:= \sum_{j=1}^n(-1)^{\frac{\nu_0}{\nu_j}} a_j X_j^{2 \frac{\nu_0}{\nu_j}}, \quad \text{with}\,\, a_1, a_2, \ldots, a_n>0
	\end{equation}
	for any strong Malcev basis $\{X_1, X_2, \ldots, X_n\}$ of the Lie algebra $\g$ and any common multiple $\nu_0$ of $\nu_1, \nu_2, \ldots, \nu_n,$ are positive Rockland operators of homogeneous degree $\nu=2\nu_0.$ It is easy to see that if $\mathcal{R}$ is a positive Rockland operator, then its powers $\mathcal{R}^k,\, k \in \mathbb{N},$ and complex conjugate $\overline{\mathcal{R}}$ are also Rockland operators.  
	
	{\it Throughout this paper, we will always assume that a Rockland operator is always positive and essentially self-adjoint on $L^2(\G).$}

	In the stratified case, assume that $\{X_1, X_2, \ldots, X_{n_1}\}$ is a basis of the first stratum $\g_1$ of the stratified Lie algebra. Then any left-invariant {\it sub-Laplacian} (with geometers sign convention) on $\G$
	\begin{align}\label{stratified}
		\mathcal{L}_\G:=-(X_1^2+X_2^2+\cdots+X_{n_1}^2)
	\end{align}is a positive Rockland operator of the homogeneous degree $\nu=2.$ 
	On $\g =\mathbb{R}^n,$ with the trivial stratification and canonical family of dilation $D_r(x)=rx, \,\,r>0,$ on the group $(\mathbb{R}^n, +)$, the Laplace operator  $-\Delta_x:=-\sum_{i=1}^n \partial_{x_i}^2$ is a particular case of a positive sub-Laplacian. By equipping the group $\G=(\mathbb{R}^n, +)$ with another  isotropic or anisotropic family of dilations with the dilations' weights $\mathbb{R}^n \ni (\nu_1, \nu_2,\ldots, \nu_n) \neq (1, 1, \ldots, 1)$  determined by the canonical basis of $\g =\mathbb{R}^n,$ the operator 
	\begin{equation}
		\mathcal{R}:= \sum_{j=1}^n(-1)^{\frac{\nu_0}{\nu_j}} a_j \partial_j^{2 \frac{\nu_0}{\nu_j}}, \quad \text{with}\,\, a_1, a_2, \ldots, a_n>0
	\end{equation}
	is a positive Rockland operator of homogeneous degree $\nu:=2\nu_0$ on $\G=(\mathbb{R}^n, +)$ provided  $\nu_0$ is any common multiple of $\nu_1, \nu_2, \ldots, \nu_n.$

	\subsection{Fourier transform on graded Lie groups}
	
	One of the important tools to deal with PDEs on graded Lie groups is the operator-valued group Fourier transform on $\G$. The group Fourier transform $\mathcal{F}_{\G   }(f)(\pi):\mathcal{H}_\pi \rightarrow \mathcal{H}_\pi$ of $f\in \mathcal{S}(\G)\cong \mathcal{S}(\mathbb{R}^n), $  at $\pi\in\widehat{\G},$ is a linear mapping that can be represented by an infinite matrix once we choose a basis for
	the Hilbert space $\mathcal{H}_\pi,$  and defined by 
	\begin{equation} \label{gft}
		\mathcal{F}_{\G   }(f)(\pi)=\widehat{f}(\pi):=\int_{\G}f(x)\pi(x)^*dx = \int_{\G} f(x)\pi(x^{-1})\,dx.
	\end{equation}
	For $f \in L^2(\G),$ the operator $\widehat{f}(\pi)$ is a Hilbert-Schmidt operator on $\mathcal{H}_{\pi}$ for each $\pi \in \widehat{G}.$ Moreover, there exists a measure $\mu$ on $\widehat{G}$ such that the following inversion formula 
	$$f(x)= \int_{\widehat{\G}} \text{Tr}(\pi(x)\widehat{f}(\pi)) d\mu(\pi)$$
	holds for every $f \in \mathcal{S}(\G)$ and $x \in \G.$
	
	Additionally, the following {\it Plancherel identity} is also true for $f \in \mathcal{S}(\G):$
	\begin{equation}
		\int_{\G} |f(x)|^2 \,dx =\int_{\widehat{G}} \|\widehat{f}(\pi)\|_{\text{HS}(\mathcal{H}_\pi)}^2\, d\mu(\pi).
	\end{equation}
	Furthermore, the Fourier transform $\mathcal{F}_{\G   }$ extends uniquely to a unitary isomorphism from $L^2(\G)$ onto the space $L^2(\widehat{\G}),$ where the space $L^2(\widehat{\G})$ is defined as the direct integral of Hilbert spaces of measurable fields of operators 
	$$L^2(\widehat{\G}):=\int_{\widehat{G}}^{\oplus} \text{HS}(\mathcal{H}_\pi) d\mu(x)$$
	with the norm 
	$$\|\tau\|_{L^2(\widehat{\G})}= \left( \int_{\widehat{\G}} \|\tau_\pi\|_{\text{HS}(\mathcal{H}_\pi)}^2\, d\mu(\pi) \right)^{\frac{1}{2}}.$$
	The measure $\mu$  is called the {Plancherel measure} on $\widehat{\G}$.
	
	Moreover, for any $f \in L^2(\mathbb{G})$,  we have 
	$$
	\mathcal{F}_{\mathbb G}(\mathcal{R} f)(\pi)=\pi(\mathcal{R})\widehat{f}(\pi).
	$$
	The authors in \cite{spectrum}  proved that the spectrum of the operator $\pi(\mathcal{R})$ with $\pi\in \widehat{ \mathbb{G}}\backslash \{1\}$, is discrete and lies in $(0, \infty)$.  Thus  we can choose an orthonormal basis for $\mathcal{H}_\pi$  such that the infinite matrix associated
	to the self-adjoint operator $\pi(\mathcal{R})$ has the  following  representation 
	\begin{align}\label{matrix}
		\pi(\mathcal{R})=\left(\begin{array}{cccc}
			\pi_1^2 & 0 & \cdots & \cdots \\
			0 & \pi_2^2 & 0 & \cdots \\
			\vdots & 0 & \ddots & \\
			\vdots & \vdots & & \ddots
		\end{array}\right)
	\end{align}
	where $\pi_i, \, i=1,2,\ldots,$ are strictly positive real numbers and  $\pi\in \widehat {\mathbb{G}}\backslash \{1\}$. %In the sequel, when we write $\widehat{f}(\pi)_{m, k}$, we will be using the same basis in the representation space $\mathcal{H}_\pi$ as the one giving (2.4). 
	
	\subsection{Sobolev spaces on graded Lie groups and interpolation inequalities}
	The Sobolev spaces on graded Lie groups were systematically studied by Fischer and the third author in \cite{RF17, Fischer}. 
	
	The inhomogeneous Sobolev spaces $H ^s(\G):=H^s_{\mathcal{R}}(\G), s \in \mathbb{R}$, associated to  positive Rockland operator  $\mathcal{R}$ of homogeneous degree $\nu$, is defined  as
	$$
	H ^s\left(\mathbb{G}\right):=\left\{f \in \mathcal{D}^{\prime}\left(\mathbb{G}\right):(I+\mathcal{R})^{s / \nu} f \in L^2\left(\mathbb{G}\right)\right\},
	$$
	with the norm $$\|f\|_{H ^s\left(\mathbb{G}\right)}:=\left\|(I+\mathcal{R})^{s / \nu} f\right\|_{L^2\left(\mathbb{G}\right)}.
	$$  
	Similarly, we define the homogeneous Sobolev space  $ \dot{H}^{p, s}_{\mathcal{R}}(\mathbb{G}):=\dot{H}^{p, s}(\mathbb{G})$ on $\G$  as the space of all  $f\in \mathcal{D}'(\mathbb{G})$ such that 
	$\mathcal{R}^{{s}/{\nu}}f\in L^p(\mathbb{G})$ with the norm
	$$\|f\|_{\dot{H}^{p, s}(\mathbb{G})}:=\left\|\mathcal{R}^{s / \nu} f\right\|_{L^p\left(\mathbb{G}\right)}.
	$$  
	The reason for omitting the subscript $\mathcal{R}$ is that 	%  Here we want to emphasize that
	these Sobolev spaces are independent of the choice of a Rockland operator $\mathcal{R}$.    
	
	Let  $\mathbb  G$ be a graded Lie group with homogeneous dimension $Q$.	  Then we use the following inequalities developed in \cite{30,Fischer, RF17}  throughout this work.

	\begin{itemize}
		\item  {\bf Sobolev inequality} \cite{Fischer, RF17}:   Let $s>0$ and $1<p<q<\infty$ be such that 	$$ 	\frac{s}{Q}=\frac{1}{p}-\frac{1}{q}.$$ Then   \begin{align}\label{eq177}
			\|f\|_{{L}^{q}(
				\mathbb G)}  \lesssim  \|f\|_{{\dot H} ^{p,s}(\mathbb G)}\simeq   \|\mathcal{R}^{\frac{s}{\nu}}f\|_{L^p(\mathbb G)}. 	\end{align}
		%\item For $s  \geq 0$, we also have 
		%$$\|f\|_{{L}^{2}(
			%\mathbb G)}  \lesssim  \|f\|_{{\dot H} ^{p,s}(\mathbb G)}$$  
		%\item Fractional Gagliardo-Nirenbeng Inequality: Let $	p, q, r \in(1, \infty)$ and  $0<a<b$. Then 
		%	$$\|f\|_{{\dot H} ^{p, a}(\mathbb G)} \lesssim \|f\|_{L^q(\mathbb G)}^\theta\| \|_{\dot{H} ^{r, b}(\mathbb G)}^{1-\theta},$$
		%for all $f\in L^q (\mathbb{G})\cap   \dot {H} ^{r, b}(\mathbb{G})$, 	
		%	where $\theta=1-\frac{a}{b}$ and $			\frac{1}{p}=\frac{\theta}{q}+\frac{1-\theta}{r}.$
		\item {\bf Gagliardo-Nirenberg  inequality} \cite{30}: 
		Let 
		$				s\in(0,1], 1<r<\frac{Q}{s},\text { and }~  2 \leq q \leq \frac{rQ}{Q-sr} 
		.$	Then  
		\begin{align}\label{eq16}
			\|u\|_{L^q(\mathbb G)} \lesssim\|u\|_{{\dot H} ^{r,s}(\mathbb G)}^\theta\|u\|_{L^2(\mathbb G)}^{1-\theta},%\simeq\left\|\mathcal{L} u\right\|_{L^r(\mathbb{G})}^\theta\|u\|_{L^p(\mathbb{G})}^{1-\theta},  
		\end{align}
		for $\theta=\left(\frac{1}{2}-\frac{1}{q}\right)/{\left(\frac{s}{Q}+\frac{1}{2}-\frac{1}{r}\right)}\in[0,1]$, provided that $\frac{s}{Q}+\frac{1}{2}\neq \frac{1}{r}$. % If $\frac{a}{Q}+\frac{1}{p}-\frac{1}{r}=0$, we have $p=q=\frac{r Q}{Q-a r}$, in which case (4.4) holds for any $0 \leq s \leq 1$.  
		
	\end{itemize}
	
	\section{Philosophy of the  approach}\label{sec3}
	In this section, we will describe the approach to prove the main results of this paper.  First, we provide the notion of mild solutions to (\ref{eq0010}) in our framework. Consider the inhomogeneous system 
	\begin{align} \label{eq001000}
		\begin{cases}
			u_{tt}+\mathcal{R}u +u_{t} =F(t,x), & x\in  \mathbb{G},~t>0,\\
			u(0,x)=  u_0(x),  & x\in  \mathbb{G},\\ u_t(0, x)=  u_1(x), & x\in  \mathbb{G}.
		\end{cases}
	\end{align}
	By applying Duhamel’s principle, the solution to the above system can be written as 
	$$u(t, x)=   u_{0} *  E_{0}(t, x)+  u_{1} *  E_{1}(t, x)+\int_{0}^{t}F(s, x) *  E_{1}(t-s, x) \;ds,$$
	where   $*$ denotes the  group convolution product on $ \mathbb G$ with respect to the $x$ variable, and  $E_{0}$ and $E_{1}$  represent   the propagators to (\ref{eq001000}) in  the homogeneous case $F=0$   with initial data $\left(u_{0}, u_{1}\right)=\left(\delta_{0}, 0\right)$ and $\left(u_{0}, u_{1}\right)=$ $\left(0, \delta_{0}\right)$, respectively. Our main interest lies in considering power-type nonlinearities, that is, $F(t,x)=|u(t, x)|^p$, and this will always be the case throughout the paper.

	A function   $u$ is  said to be a {\it mild solution} to (\ref{eq0010})  on $[0, T]$ if $u$ is a fixed point for  the    integral operator $N: u \in X_s(T) \mapsto N u(t, x) ,$ given  by 
	\begin{align}\label{f2inr} 
		N u(t, x):=u^{\text{lin}}(t, x) + u^{\text{non}}(t, x),
	\end{align}	in the energy evolution space $X_s(T) \doteq \mathcal{C}\left([0, T], H^{1}(\mathbb{G})\right),  $ 
	equipped with the norm
	\begin{align}
		\|u\|_{X_s(T)}&:=\sup\limits_{t\in[0,T]}\left ( (1+t)^{\frac{\gamma}{\nu}} \|u(t,\cdot)\|_{L^2}+(1+t)^{\frac{1+\gamma}{\nu}}\|u(t,\cdot)\|_{ \dot H^1}\right )
	\end{align}
	with $\gamma>0,$
	where   $$u^{\text{lin}}(t, x)=   u_{0} *  E_{0}(t, x)+  u_{1} *  E_{1}(t, x)$$  
	is the solution to the corresponding linear Cauchy problem (\ref{eq0010}), and $u^{\text{non}}$ is an integral operator with the following representation
	$$ u^{\text{non}}(t, x)= \int_{0}^{t} F(s, x) *  E_{1}(t-s, x)  \;ds.$$
	According to the Duhamel’s principle, we will   prove the global-in-time existence and uniqueness of small data Sobolev solutions of low regularity to the semilinear damped wave equation (\ref{eq0010}) with
	the help of the Banach's  fixed point theorem argument.  We find a unique fixed point (say) $u^*$ of the operator $N$, which means $u^*=N u^* \in X_s(T)$ for all positive $T$. More precisely, to find such a unique fixed point,   we will establish two crucial inequalities of the form
	\begin{align}\label{Banach1}
		\|N u\|_{X_s(T)} & \leq C  \left\|\left(u_0, u_1\right)\right\|_{\mathcal{A}^s}+\|u\|_{X_s(T)}^p,  \end{align}
	and \begin{align}\label{Banach2}
		\|N u-N v\|_{X_s(T)} & \leq C\|u-v\|_{X_s(T)}\left[  \|u\|_{X_s(T)}^{p-1}+\|v\|_{X_s(T)}^{p-1}\right],
	\end{align}
	for any $u, v \in X_s(T)$ with initial data space $\mathcal{A}^s:=\ (H^1 \cap \dot{H}^{-\gamma}\ ) \times\ (L^2 \cap \dot{H}^{-\gamma}\ )$ and the positive constant constant $C$ is independent of $T$. 
	We will consider sufficiently small $\left\|\left(u_0, u_1\right)\right\|_{\mathcal{A}^s}<\varepsilon $ so that combining (\ref{Banach1}) with (\ref{Banach2}) we can apply  Banach's fixed point theorem to ensure that there exists a global-in-time small data unique Sobolev solution $u^*=Nu^* \in X_s(T)$ for all  $T>0$, which also gives the solution to (\ref{eq0010}). Here we want to note that the treatment of the power-type nonlinear term     is based on the applications of the Hardy-Littlewood-Sobolev inequality and the  Gagliardo-Nirenberg inequality on the graded Lie group $\mathbb G$.
	
	In the proof of  the global existence result  for $p=p_{\text{crit}}$ in the next section, the following integral inequality will be useful:
	Let $\alpha, \beta \in \mathbb{R}$. Then
	$$
	\int_0^t(1+t-\kappa)^{-\alpha}(1+\kappa)^{-\beta} \;d\kappa  \lesssim \begin{cases}(1+t)^{-\min \{\alpha, \beta\}} & \text{ if } \max \{\alpha, \beta\}>1; \\{}\\ (1+t)^{-\min \{\alpha, \beta\}} \log (2+t) & \text { if } \max \{\alpha, \beta\}=1; \\{}\\ (1+t)^{1-\alpha-\beta} & \text { if } \max \{\alpha, \beta\}<1.\end{cases}
	$$

	\section{ Global existence for critical higher order hypoelliptic damped wave equation} \label{sec4}	
	This section is devoted to proving the global-in-time existence of small data Sobolev solutions to (\ref{eq0010})  of lower regularity in the critical case.  Before going to  investigate the global-in-time existence result in the critical case,  we first present the following proposition. 
	\begin{prop}\label{Prop}
		Let ${\mathbb{G}}$ be a graded Lie group of homogeneous dimension $Q$ and let $\mathcal{R}$ be a positive Rockland operator of homogeneous degree $\nu \geq 2.$
		Let $s \in (0, 1]$ and  $(u_0, u_1) \in \left( H^s \cap L^m   \right) \times \left(L^2   \cap L^m  \right) $ for some $m \in(1,2]$. Then the solution $u \in \mathcal{C}^1\left([0, \infty), H^s  \right)$ of \eqref{Linear-system} satisfies  the following estimates:
		\begin{align}\label{s=00}
			\|u(t, \cdot)\|_{L^2  } \leq C(1+t)^{-\frac{Q}{\nu}\left(\frac{1}{m}-\frac{1}{2}\right)}  \left(\left\|u_0\right\|_{L^2   \cap L^m  }+\left\|u_1\right\|_{H ^{-1}   \cap L^m  }\right),\end{align}
		and	 
		\begin{align}\label{s=11}
			\left\|\mathcal{R}^{\frac{s}{\nu}} u(t, \cdot)\right\|_{L^2  }  \leq C(1+t)^{-\frac{Q}{\nu}\left(\frac{1}{m}-\frac{1}{2}\right)-\frac{s}{\nu}}\left(\left\|u_0\right\|_{H^s \cap L^m}+\left\|u_1\right\|_{H^{s-1}    \cap L^m  }\right).
		\end{align}
	\end{prop} 
	\begin{proof} The case $m=2$ directly follows from \cite[Theorem 1.1]{DKMR24} by taking $\gamma=0.$
		For  $\gamma\in (0, \frac{Q}{2})$ and for $1<m<2$  with 	$$ 	\frac{\gamma}{Q}=\frac{1}{m}-\frac{1}{2},$$ from Sobolev inequality (\ref{eq177}), we have  \begin{align*}
			\|u(t, \cdot )\|_{{\dot H} ^{-\gamma }}\simeq	\| \mathcal{R}^{-\frac{\gamma}{\nu}} u(t, \cdot )\|_{{L}^{2}}  \lesssim     \| u(t, \cdot )\|_{L^m},	\end{align*}
		and hence
		\begin{align}\label{eq1177}
			\|u(t, \cdot )\|_{H^{s-1}  \cap {\dot H}^{-\gamma }} \lesssim     \| u(t, \cdot )\|_{H^{s-1}  \cap L^m}.
		\end{align}
		Now using (\ref{eq1177}), from \cite[Theorem 1.1]{DKMR24} with $s\in (0, 1]$, we get 
		$$
		\begin{aligned}
			\left\|\mathcal{R}^{\frac{s}{\nu}} u(t, \cdot)\right\|_{L^2} &= C\|u(t, \cdot)\|_{\dot{H}^s} \\
			&\lesssim(1+t)^{-\frac{s+\gamma}{\nu}}\left(\left\|u_0\right\|_{H^s \cap\dot {H}^{-\gamma}}+\left\|u_1\right\|_{H^{s-1} \cap \dot{H}^{-\gamma}}\right) \\
			& \lesssim(1+t)^{-\frac{Q}{\nu}\left(\frac{1}{m}-\frac{1}{2}\right)-\frac{s}{\nu}}\left(\left\|u_0\right\|_{H^s \cap L^m}+\left\|u_1\right\|_{H^{s-1}    \cap L^m  }\right) 
		\end{aligned}
		$$
		and 
		$$
		\begin{aligned}
			\|u(t, \cdot)\|_{L^2} 
			&  \lesssim(1+t)^{-\frac{\gamma}{\nu}}\left(\left\|u_0\right\|_{L^2 \cap\dot {H}^{-\gamma}}+\left\|u_1\right\|_{H^{-1} \cap \dot{H}^{-\gamma}}\right) \\
			&  \lesssim(1+t)^{-\frac{Q}{\nu}\left(\frac{1}{m}-\frac{1}{2}\right)} \left(\left\|u_0\right\|_{L^2   \cap L^m  } +\left\|u_1\right\|_{H^{-1}   \cap L^m  }\right)
		\end{aligned}
		$$
		for all $t>0.$
	\end{proof}
	
	\begin{rem} \label{rem4.2}
		Note that, the estimate (\ref{s=00}) can also be written as (using Sobolev inequality)    	\begin{align*} 
			\|u(t, \cdot)\|_{L^2  } &\leq C (1+t)^{-\frac{Q}{\nu}\left(\frac{1}{m}-\frac{1}{2}\right)}  \left(\left\|u_0\right\|_{L^2  \cap L^m  }+\left\|u_1\right\|_{H ^{-1}   \cap L^m  }\right)\\
			& \leq C'(1+t)^{-\frac{Q}{\nu}\left(\frac{1}{m}-\frac{1}{2}\right)}  \left(\left\|u_0\right\|_{H^s \cap L^m}+\left\|u_1\right\|_{H^{s-1}    \cap L^m  }\right).
		\end{align*}
		Then combining it with (\ref{s=11}), for $j=0, 1$, one can write 
		\begin{align}\label{s=01}
			\left\|\mathcal{R}^{\frac{js}{\nu}} u(t, \cdot)\right\|_{L^2  }  \leq C(1+t)^{-\frac{Q}{\nu}\left(\frac{1}{m}-\frac{1}{2}\right)-\frac{js}{\nu}}\left(\left\|u_0\right\|_{H^s \cap L^m}+\left\|u_1\right\|_{H^{s-1}    \cap L^m  }\right).
		\end{align}
	\end{rem}
	Now we are in a position to prove the global-in-time existence of small data Sobolev solutions to (\ref{eq0010})  of lower regularity in the critical case as follows. 
	\begin{theorem}\label{well-posed}  Let ${\mathbb{G}}$ be a graded Lie group of homogeneous dimension $Q$ and let $\mathcal{R}$ be a positive Rockland operator of homogeneous degree $\nu \geq 2.$ Let $s \in (0, 1].$ Assume that 
		$$ \begin{cases}
			\gamma \in\left(0, \frac{Q}{2}\right) 
			\quad  \hspace{1.2cm} \text{if}  \quad Q=1, 2;\\
			\gamma \in\left(0, \min\{\frac{Q}{2}, \tilde{\gamma} \}\right) 
			\quad  \hspace{1.4cm} \text{if}  \quad Q=3, 4;\\
			\gamma \in \left( \frac{Q\nu}{2s}-\frac{Q}{2} -\nu, \min\{\frac{Q}{2}, \tilde{\gamma} \}\right) 
			\quad  \hspace{0.2cm} \text{if}  \quad Q=5, 6,
		\end{cases}$$ where $\tilde{\gamma}$ denotes the positive root of  the quadratic equation $2 \tilde{\gamma}^2+Q \tilde{\gamma}-\nu Q=0$, i.e., $\tilde{\gamma}= \frac{-Q+\sqrt {Q^2+8\nu Q}}{4}$ for $Q\geq 3$.   Also,  let the exponent $p$ satisfy 
		\begin{align}\label{eq24}
			p=p_{\text {Crit }}(Q, \gamma, \nu):=1+\frac{2\nu}{Q+2\gamma}.  		\end{align}
		Then, there exists a small positive constant $\varepsilon_0$ such that for any $\left(u_0, u_1\right) \in \mathcal{A}^{s }:=\ (H^s \cap \dot{H}^{-\gamma}\ ) \times\ (L^2 \cap \dot{H}^{-\gamma}\ )$ satisfying $\left\|\left(u_0, u_1\right)\right\|_{\mathcal{A}^{s}}=\varepsilon \in\left(0, \varepsilon_0\right]$, the Cauchy problem for the semilinear damped wave equation (\ref{eq0010}) has a uniquely determined Sobolev solution
		$$
		u \in \mathcal{C}\left([0, \infty), H^s\right).
		$$

	\end{theorem} 
	\begin{proof}
		Our main aim is to prove the following 	two crucial inequalities 
		\begin{align*}
			\|N u\|_{X_s(T)} & \lesssim\left\|\left(u_{0}, u_{1}\right)\right\|_{ \mathcal{A}^s}+\|u\|_{X_s(T)}^{p}, \\	\|N u-N v\|_{X_s(T)} &\lesssim \|u-v\|_{X_s(T)}\left(\|u\|_{X_s(T)}^{p-1}+\|v\|_{X_s(T)}^{p-1}\right),
		\end{align*} 	     
		where the energy evolution space $X_s(T) \doteq \mathcal{C}\left([0, T], H^{s}(\mathbb{G})\right)$ is 
		equipped with the norm
		\begin{align} \label{Evolution norm}
			\|u\|_{X_s(T)}&:=\sup\limits_{t\in[0,T]}\left ( (1+t)^{\frac{\gamma}{\nu}} \|u(t,\cdot)\|_{L^2}+(1+t)^{\frac{s+\gamma}{\nu}}\|u(t,\cdot)\|_{ \dot H^s}\right ).
		\end{align}

		Now from   the estimate \cite[Theorem 1.1]{DKMR24}, for  $s \geq 0$ and $\gamma \in \mathbb{R}$ such that $s+\gamma \geq  0$,  we recall that 
		\begin{equation} \label{4er}
			\|u(t, \cdot)\|_{\dot{{H}} ^s(\mathbb{G})} \lesssim(1+t)^{-\frac{s+\gamma}{\nu}}\left(\left\|u_0\right\|_{H^s  \cap \dot {{H}} ^{-\gamma} }+\left\|u_1\right\|_{H ^{s-1}  \cap \dot{{H}} ^{-\gamma} }\right) .
		\end{equation}
		Now depending on the value of $\gamma $, from \eqref{4er},  	using  the Sobolev embedding   $L^2\subset H ^{s-1}$ for $s\leq 1$ and $ H ^{s} \subset L^2$ for $s\geq0$, we have the following   crucial estimates for Sobolev solutions to the linear Cauchy problem.
		%\begin{itemize}
		%	\item For   $\gamma=0$ and  $s \in [0, 1]$, we get 
		%	\begin{align}\label{eq15}		\|u(t, \cdot)\|_{{\dot{H}} ^s}&		\lesssim(1+t)^{-\frac{s}{\nu}}\left(\left\|u_0\right\|_{H^s  \cap L^2 }+\left\|u_1\right\|_{H ^{s-1}  \cap L^2}\right) .%&\lesssim(1+t)^{-\frac{s}{\nu}}\left(\left\|u_0\right\|_{H ^s }+\left\|u_1\right\|_{L^2}\right).	\end{align}
			%	\item 
			For $\gamma>0$ and  $s \in [0, 1]$, 	 	we have 
			\begin{align}\label{eq14}
				\|u(t, \cdot)\|_{\dot{{H}} ^s} \lesssim(1+t)^{-\frac{s+\gamma}{\nu}}\left(\left\|u_0\right\|_{ H^s \cap {\dot{H}} ^{-\gamma}}+\left\|u_1\right\|_{L^2 \cap {\dot{H}} ^{-\gamma}}\right).
			\end{align}
			%\end{itemize}
			Thus 
			\begin{align}\label{eq1444}
				\|u(t, \cdot)\|_{L^2} &\lesssim(1+t)^{-\frac{\gamma}{\nu}}\left(\left\|u_0\right\|_{ L^2 \cap {\dot{H}} ^{-\gamma}}+\left\|u_1\right\|_{L^2 \cap {\dot{H}} ^{-\gamma}}\right) 
			\end{align}
			and for $s\in (0, 1]$
			\begin{align}\label{eq14444}
				\|u(t, \cdot)\|_{\dot{{H}} ^s} \lesssim(1+t)^{-\frac{s+\gamma}{\nu}}\left(\left\|u_0\right\|_{H^s  \cap \dot {{H}} ^{-\gamma} }+\left\|u_1\right\|_{H ^{s-1}  \cap \dot{{H}} ^{-\gamma} }\right) .
			\end{align}
			Now from (\ref{eq1444}) and   (\ref{eq14444}),   we can write
			\begin{align*}
				(1+t)^{\frac{s+\gamma}{\nu}}\|u(t, \cdot)\|_{{\dot {H}} ^s} +(1+t)^{\frac{\gamma}{\nu}}\|u(t, \cdot)\|_{L^2} \lesssim \left(\left\|u_0\right\|_{H^s \cap {\dot{H}} ^{-\gamma}}+\left\|u_1\right\|_{L^2 \cap {\dot {H}} ^{-\gamma}}\right)=\left\|\left(u_{0}, u_{1}\right)\right\|_{ \mathcal{A}^s}.
			\end{align*}
			Thus from above,   the  solution $u^{\text{lin}}$  
			to the linear problem (\ref{eq0010}) satisfies   \begin{align}\label{2number100}
				\|u^{\text{lin}}\|_{X_s(T)} \leq C_1   \left\|\left(u_{0}, u_{1}\right)\right\|_{ \mathcal{A}^s} 
			\end{align} 
			where $C_1>0$ is independent of $T$.
			
			Now we will   evaluate $
			\| u^{\text{non}}-v^{\text{non}}\|_{X_s(T)}$ and try to prove that $$	\|u^{\text{non}}-v^{\text{non}}\|_{X_s(T)} \leq C_2 \|u-v\|_{X_s(T)}\left(\|u\|_{X_s(T)}^{p-1}+\|v\|_{X_s(T)}^{p-1}\right),$$  for some $C_2>0$ independent of $T$. Notice that 
			$$
			u^{\text{non}}-v^{\text{non}}(t, \cdot )=\int_0^t E_1(t-\kappa, \cdot) *\left(|u|^p-|v|^p\right)  {d} \kappa.
			$$
			In order to estimate $\|u^{\text{non}}-v^{\text{non}}\|_{X_s(T)}$, we   first evaluate  $\| \mathcal{R}^{\frac{sj}{\nu}}(u^{\text{non}}-v^{\text{non}}) \|_{L^2}$ for $j=0, 1$ with the help of  Proposition \ref{Prop}.
			
			Now applying the Gagliardo-Nirenberg  inequality (\ref{eq16}), for any $u\in X_s(T)$, we get
			$$	\left \|u(t, \cdot)\right\|_{L^{q}}\lesssim\|u(t, \cdot)\|_{\dot {H}^s}^{\theta} \|u(t, \cdot)\|_{L^2 }^{(1-\theta)}$$
			for $t\in [0, T]$ with $\theta=\frac{Q}{s}(\frac{1}{2}-\frac{1}{q})\in [0, 1]$. The above inequality also can be written as  
			\begin{align}\label{eqq21}\nonumber
				\left \|u(t, \cdot)\right\|_{L^{q}}
				&= (1+t)^{-\frac{1}{\nu}\left[s \theta+\gamma \right ]}  \left\{ (1+t)^{\frac{(s+\gamma)}{\nu} } \|u(t, \cdot)\|_{\dot {H}^s} \right\}^{\theta}  \left\{ (1+t)^{\frac{\gamma}{\nu} } \|u(t, \cdot)\|_{L^2}\right\}^{(1-\theta)}\\ \nonumber
				&\leq (1+t)^{-\frac{1}{\nu}\left[s \theta+\gamma \right ]}        \left\| u\right\|_{ X_s(T)}\\
				&= (1+t)^{-\frac{1}{\nu}\left[ {Q}(\frac{1}{2}-\frac{1}{q})+\gamma \right ]}        \left\| u\right\|_{ X_s(T)},
			\end{align}
			for $2\leq  q\leq \frac{2Q}{Q-2s}. 		$			
			Now for $j=0,1$, we have 
			$$
			\begin{aligned}
				\| \mathcal{R}^{\frac{sj}{\nu}}(u^{\text{non}}-v^{\text{non}})(t, \cdot ) \|_{L^2}&=\left\|\int_0^t \mathcal{R}^{\frac{sj}{\nu}} E_1(t-\kappa, \cdot) *\left(|u|^p-|v|^p\right)  {d} \kappa\right \|_{L^2}\\
				&\leq  \int_0^t \left\| \mathcal{R}^{\frac{sj}{\nu}} E_1(t-\kappa, \cdot) *\left(|u|^p-|v|^p\right) \right \|_{L^2} {d} \kappa.
			\end{aligned}
			$$		
			Next, note that   $$\tilde{u}(t,x):=(    E_1(t, \cdot) *|u|^p  )(x)$$ is also a solution to the linear system (\ref{Linear-system}) with  initial data $u_0=0$, $u_1=|u|^p $.  To estimate    Duhamel’s term on the interval  $[0, t]$, from  Proposition  \ref{Prop} (also see \eqref{s=01} of Remark \ref{rem4.2} ) with $u_0=0$, $u_1=|u|^p-|v|^p$ and 	 by  the Sobolev embedding   $L^2\subset H ^{s-1}$ for $s< 1$, we obtain  
			\begin{align}\label{1111}\nonumber
				\| \mathcal{R}^{\frac{sj}{\nu}}(u^{\text{non}}-v^{\text{non}}) (t, \cdot )\|_{L^2}&\leq   \int_0^t (1+t-\kappa )^{-\frac{Q}{\nu}\left(\frac{1}{m}-\frac{1}{2}\right)-\frac{sj}{\nu}} \left\||u|^p-|v|^p\right\|_{H^{s-1}  \cap L^m  }   {d} \kappa \nonumber \\
				&\leq   \int_0^t (1+t-\kappa )^{-\frac{Q}{\nu}\left(\frac{1}{m}-\frac{1}{2}\right)-\frac{sj}{\nu}} \left\||u|^p-|v|^p\right\|_{L^2   \cap L^m  }   {d} \kappa\\\nonumber
				&  \leq \int_0^{t / 2}(1+t-\kappa )^{-\frac{Q}{\nu}\left(\frac{1}{m}-\frac{1}{2}\right)-\frac{sj}{\nu}} \|(|u|^p-|v|^p)(\kappa, \cdot)\|_{L^2 \cap L^{\frac{2}{p}} }\;d \kappa \\
				& \quad+\int_{t / 2}^t(1+t-\kappa )^{-\frac{sj}{\nu}} \|(|u|^p-|v|^p)(\kappa, \cdot)\|_{L^2} \;d \kappa,
			\end{align}
			where we consider  $m = \frac{2}{p} $ for $p\in (1, 2)$ and $m = 1$ for $p \geq 2$, however,  we use $m = 2$ for  the interval $ [\frac{t}{2}, t]$. 
			Now using the fact that  $$||u|^p-|v|^p|\leq p |u-v|(|u|^{p-1}+|v|^{p-1}) $$ along with      H\"older's inequality, we   obtain
			%			$$
			%			\begin{aligned}
				%				& \|(|u|^p-|v|^p|)(\tau, \cdot)\|_{L^2} \\
				%				& \quad \leq C\|(u-v)(\tau, \cdot)\|_{L^{2 p}}\left(\|u(\tau, \cdot)\|_{L^{2 p}}^{p-1}+\|v(\tau, \cdot)\|_{L^{2 p}}^{p-1}\right) \\
				%				& \quad \leq C(1+\tau)^{-\frac{q}{4}(p-1)-\frac{\tau}{2} p} D(u, v)=C(1+\tau)^{-1-\frac{-}{2}} D(u, v)
				%			\end{aligned}
			%			$$
			\begin{align}\label{eq25}
				\left\||u(\kappa, \cdot)|^p-|v(\kappa, \cdot)|^p\right\|_{L^2} \lesssim\|u(\kappa, \cdot)-v(\kappa, \cdot)\|_{L^{2 p}}\left(\|u(\kappa, \cdot)\|_{L^{2 p}}^{p-1}+\|v(\kappa, \cdot)\|_{L^{2 p}}^{p-1}\right) .
			\end{align}
			Now we have to  estimate three terms on the right-hand side of the previous inequality (\ref{eq25}).    Applying  the Gagliardo-Nirenberg inequality (\ref{eq16})   for   each   term that appeared on the right-hand side of (\ref{eq25}), we get 
			\begin{align}\label{Final2}\nonumber
				& 	\left \| u(\kappa, \cdot)-v(\kappa, \cdot)\right\|_{L^{2p}}\\\nonumber
				&\lesssim\| u(\kappa, \cdot)-v(\kappa, \cdot)\|_{\dot {H}^s}^{\theta_1} \|u(\kappa, \cdot)\|_{L^2 }^{(1-\theta_1)}\\\nonumber
				&= (1+\kappa)^{-\frac{1}{\nu}\left(  {\gamma}+s {\theta_1}\right)  }  \left\{ (1+\kappa)^{\frac{(s+\gamma)}{\nu} } \|u(\kappa, \cdot)-v(\kappa, \cdot)\|_{\dot {H}^s} \right\}^{\theta_1}   \left\{ (1+\kappa)^{\frac{\gamma}{\nu} } \|u(\kappa, \cdot)-v(\kappa, \cdot)\|_{L^2 }\right\}^{(1-\theta_1)}\\
				&\lesssim (1+\kappa)^{-\frac{1}{\nu}\left(  {\gamma}+ s{\theta_1}\right)  }\left\| u-v \right\|_{ X_s(T)}
			\end{align}
			and 
			\begin{align}\label{Final1}\nonumber
				&	 	\left \| u(\kappa, \cdot)\right\|_{L^{2p}}^{p-1}\\\nonumber
				&\lesssim\| u(\kappa, \cdot)\|_{\dot {H}^s(\G)}^{(p-1)\theta_1} \|u(\kappa, \cdot)\|_{L^2(\G)}^{(p-1)(1-\theta_1)}\\\nonumber
				&= (1+\kappa)^{-\frac{(p-1)}{\nu}  (s\theta_1+\gamma)}  \left\{ (1+\kappa)^{\frac{(s+\gamma)}{\nu} } \|u(\kappa, \cdot)\|_{\dot {H}^s(\G)} \right\}^{(p-1)\theta_1}   \left\{ (1+\kappa)^{\frac{\gamma}{\nu} } \|u(\kappa, \cdot)\|_{L^2(\G)}\right\}^{(p-1)(1-\theta_1)}\\
				&\lesssim (1+\kappa)^{-\frac{(p-1)}{\nu}  (s\theta_1+\gamma)} \left\| u\right\|_{ X_s(T)}^{p-1},
			\end{align}
			for $\kappa\in [0, T]$ with $\theta_1=\frac{Q}{s}(\frac{1}{2}-\frac{1}{2p})\in [0, 1]$.
			Using the above estimates, (\ref{eq25}) reduces to 
			\begin{align}\label{eq444}\nonumber
				\left\||u(\kappa, \cdot)|^p-|v(\kappa, \cdot)|^p\right\|_{L^2} &\leq C  (1+\kappa)^{   -\frac{p}{\nu}  (s\theta_1+\gamma)} \left\| u-v \right\|_{ X_s(T)} \left( \left\| u\right\|_{ X_s(T)}^{p-1}+ \left\| v\right\|_{ X_s(T)}^{p-1}\right)\\\nonumber
				&= C  (1+\kappa)^{ -\frac{Q(p-1)}{2\nu}  -\frac{p\gamma}{\nu }} \left\| u-v \right\|_{ X_s(T)} \left( \left\| u\right\|_{ X_s(T)}^{p-1}+ \left\| v\right\|_{ X_s(T)}^{p-1}\right)\\ 
				&=  C  (1+\kappa)^{ -1 -\frac{\gamma}{\nu }} \left\| u-v \right\|_{ X_s(T)} \left( \left\| u\right\|_{ X_s(T)}^{p-1}+ \left\| v\right\|_{ X_s(T)}^{p-1}\right)
				%	&\leq  C  (1+\kappa)^{ -1 -\frac{\gamma}{\nu }} \left\| u-v \right\|_{ X_s(T)} \left( \left\| u\right\|_{ X_s(T)}^{p-1}+ \left\| v\right\|_{ X_s(T)}^{p-1}\right).
			\end{align}
			where in the  last step we used that $	p=p_{\text {Crit }}(Q, \gamma, \nu):=1+\frac{2\nu}{Q+2\gamma}.$ Then the second integral in (\ref{1111}) becomes
			\begin{align}\label{2222}\nonumber
				&	\int_{t / 2}^t(1+t-\kappa )^{-\frac{js}{\nu}} \|(|u|^p-|v|^p)(\kappa, \cdot)\|_{L^2} \;d \kappa\\\nonumber
				& \leq C  \left\| u-v \right\|_{ X_s(T)} \left( \left\| u\right\|_{ X_s(T)}^{p-1}+ \left\| v\right\|_{ X_s(T)}^{p-1}\right)	\int_{t / 2}^t(1+t-\kappa )^{-\frac{js}{\nu}}  (1+\kappa)^{ -1  -\frac{\gamma}{\nu }} \;d \kappa\\\nonumber
				& \leq C  (1+t)^{ -1 -\frac{\gamma}{\nu }}  \left\| u-v \right\|_{ X_s(T)} \left( \left\| u\right\|_{ X_s(T)}^{p-1}+ \left\| v\right\|_{ X_s(T)}^{p-1}\right)	\int_{t / 2}^t(1+t-\kappa )^{-\frac{js}{\nu}}  \;d \kappa\\
				&	\leq C (1+t)^{ -\frac{\gamma+js}{\nu }}   \left\| u-v \right\|_{ X_s(T)} \left( \left\| u\right\|_{ X_s(T)}^{p-1}+ \left\| v\right\|_{ X_s(T)}^{p-1}\right).
			\end{align}
			For the  first integral of (\ref{1111}),  we first consider  $p \geq 2$ so that $m=1$. Then 
			$$	\left\||u(\kappa, \cdot)|^p-|v(\kappa, \cdot)|^p\right\|_{L^1} \lesssim\|u(\kappa, \cdot)-v(\kappa, \cdot)\|_{L^{p}}\left(\|u(\kappa, \cdot)\|_{L^{p}}^{p-1}+\|v(\kappa, \cdot)\|_{L^{p}}^{p-1}\right) $$
			and proceeding similarly as in (\ref{Final2}), \eqref{Final1}, and   (\ref{eq444}) with $\theta_1=\frac{Q}{s}(\frac{1}{2}-\frac{1}{p}) $,  we get 
			\begin{align}\label{L^1} 
				\left\||u(\kappa, \cdot)|^p-|v(\kappa, \cdot)|^p\right\|_{L^1}  \leq  C  (1+\kappa)^{\frac{Q}{2\nu} -1 -\frac{\gamma}{\nu }} \left\| u-v \right\|_{ X_s(T)} \left( \left\| u\right\|_{ X_s(T)}^{p-1}+ \left\| v\right\|_{ X_s(T)}^{p-1}\right).
			\end{align}
			Then using the fact that $\gamma\in (0, \frac{Q}{2})$, i.e., $\frac{Q}{2\nu}-\frac{\gamma}{\nu}-1>-1$ along with the $L^2$ and $L^1$ norm of $ |u(\kappa, \cdot)|^p-|v(\kappa, \cdot)|^p $ from (\ref{eq444}) and (\ref{L^1}),  the first integral of \eqref{1111} becomes 
			\begin{align}\label{3333}\nonumber
				&	\int_0^{\frac{t}{2}} (1+t-\kappa )^{-\frac{Q}{2\nu} -\frac{sj}{\nu}} \left\||u(\kappa, \cdot)|^p-|v(\kappa, \cdot)|^p\right\|_{L^2   \cap L^1  }   {d} \kappa\\\nonumber
				&	=\int_0^{\frac{t}{2}} (1+t-\kappa )^{-\frac{Q}{2\nu} -\frac{sj}{\nu}}  \left(  \left\||u(\kappa, \cdot)|^p-|v(\kappa, \cdot)|^p\right\|_{L^1}+ \left\||u(\kappa, \cdot)|^p-|v(\kappa, \cdot)|^p\right\|_{L^2   }\right)   {d} \kappa\\\nonumber
				&	\lesssim  (1+t )^{-\frac{Q}{2\nu} -\frac{sj}{\nu}} \int_0^{\frac{t}{2}}    (1+\kappa)^{ -1 -\frac{\gamma}{\nu }} \left\| u-v \right\|_{ X_s(T)} \left( \left\| u\right\|_{ X_s(T)}^{p-1}+ \left\| v\right\|_{ X_s(T)}^{p-1}\right)  {d} \kappa\\\nonumber
				&	\quad\quad+ (1+t )^{-\frac{Q}{2\nu} -\frac{sj}{\nu}}  \int_0^{\frac{t}{2}}   (1+\kappa)^{\frac{Q}{2\nu} -1 -\frac{\gamma}{\nu }} \left\| u-v \right\|_{ X_s(T)} \left( \left\| u\right\|_{ X_s(T)}^{p-1}+ \left\| v\right\|_{ X_s(T)}^{p-1}\right) {d} \kappa\\\nonumber
				&	 \lesssim (1+t )^{-\frac{Q}{2\nu} -\frac{sj}{\nu}}  \int_0^{\frac{t}{2}} (1+\kappa)^{\frac{Q}{2\nu} -1 -\frac{\gamma}{\nu }} \left\| u-v \right\|_{ X_s(T)} \left( \left\| u\right\|_{ X_s(T)}^{p-1}+ \left\| v\right\|_{ X_s(T)}^{p-1}\right)  {d} \kappa\\\nonumber
				&	 = C(1+t )^{-\frac{Q}{2\nu} -\frac{sj}{\nu}}   \left\| u-v \right\|_{ X_s(T)} \left( \left\| u\right\|_{ X_s(T)}^{p-1}+ \left\| v\right\|_{ X_s(T)}^{p-1}\right) \int_0^{\frac{t}{2}} (1+\kappa)^{\frac{Q}{2\nu} -1 -\frac{\gamma}{\nu }}   {d} \kappa\\
				&	 \leq C' (1+t )^{ -\frac{sj+\gamma}{\nu}}   \left\| u-v \right\|_{ X_s(T)} \left( \left\| u\right\|_{ X_s(T)}^{p-1}+ \left\| v\right\|_{ X_s(T)}^{p-1}\right) .
			\end{align}
			Now  considering the case  $p \leq 2$ so that $m=\frac{2}{p}$, we get 
			\begin{align*}
				\left\||u(\kappa, \cdot)|^p-|v(\kappa, \cdot)|^p\right\|_{L^{\frac{2}{p}}}& \leq C\|u(\kappa, \cdot)-v(\kappa, \cdot)\|_{L^{2}}\left(\|u(\kappa, \cdot)\|_{L^{2}}^{p-1}+\|v(\kappa, \cdot)\|_{L^{2}}^{p-1}\right) \\
				& =C (1+\kappa)^{-\frac{\gamma}{\nu}  -\frac{\gamma}{\nu} (p-1)}   	\left\{  (1+\kappa)^{\frac{\gamma}{\nu} }  \|u(\kappa, \cdot)-v(\kappa, \cdot)\|_{L^{2}}   \right\}\\
				&\quad \times   \left(  (1+\kappa)^{\frac{\gamma}{\nu} (p-1)} \|u(\kappa, \cdot)\|_{L^{2}}^{p-1}+  (1+\kappa)^{\frac{\gamma}{\nu}(p-1) }\|v(\kappa, \cdot)\|_{L^{2}}^{p-1}\right) \\
				& \leq C' (1+\kappa)^{ -\frac{\gamma p}{\nu}}   \left\| u-v \right\|_{ X_s(T)} \left( \left\| u\right\|_{ X_s(T)}^{p-1}+ \left\| v\right\|_{ X_s(T)}^{p-1}\right) .
			\end{align*}
			Then using the fact that $\frac{\gamma p}{\nu}=\frac{\gamma}{\nu} \left(1+\frac{2\nu}{Q+2\gamma}\right)<1$ for $\gamma <\min\{\tilde{\gamma}, \frac{Q}{2}\}$ with $Q\geq 3,$ we get
			\begin{align}\label{4444}\nonumber
				&	\int_0^{\frac{t}{2}} (1+t-\kappa )^{-\frac{Q}{\nu}\left(\frac{p}{2}-\frac{1}{2}\right)-\frac{sj}{\nu}}  \left\||u|^p-|v|^p\right\|_{L^2   \cap L^{\frac{2}{p}} }   {d} \kappa\\\nonumber
				&	 \leq  C(1+t )^{-\frac{Q(p-1)}{2\nu} -\frac{sj}{\nu}}   \left\| u-v \right\|_{ X_s(T)} \left( \left\| u\right\|_{ X_s(T)}^{p-1}+ \left\| v\right\|_{ X_s(T)}^{p-1}\right) \int_0^{\frac{t}{2}} (1+\kappa)^{ -\frac{\gamma p}{\nu}} {d} \kappa\\
				&	 \leq C' (1+t )^{ -\frac{sj+\gamma}{\nu}}   \left\| u-v \right\|_{ X_s(T)} \left( \left\| u\right\|_{ X_s(T)}^{p-1}+ \left\| v\right\|_{ X_s(T)}^{p-1}\right) .
			\end{align}
			Thus combining all the cases   \eqref{2222}, \eqref{3333}, and \eqref{4444}, the inequality \eqref{1111} becomes
			\begin{align*}
				\| \mathcal{R}^{\frac{sj}{\nu}}(u^{\text{non}}-v^{\text{non}})(t, \cdot ) \|_{L^2}\leq A (1+t )^{ -\frac{sj+\gamma}{\nu}}   \left\| u-v \right\|_{ X_s(T)} \left( \left\| u\right\|_{ X_s(T)}^{p-1}+ \left\| v\right\|_{ X_s(T)}^{p-1}\right), 
			\end{align*}
			for some constant independent of $T.$ This  gives us  
			\begin{align}\label{final22}
				\|u^{\text{non}}-v^{\text{non}}\|_{X_s(T)} \leq A \|u-v\|_{X_s(T)}\left(\|u\|_{X_s(T)}^{p-1}+\|v\|_{X_s(T)}^{p-1}\right)
			\end{align}
			and from the definition of $N$ in \eqref{f2inr}, we have 
			\begin{align}\label{final 2}
				\|N u-N v\|_{X_s(T)}=	\|u^{\text{non}}-v^{\text{non}}\|_{X_s(T)} \leq A \|u-v\|_{X_s(T)}\left[  \|u\|_{X_s(T)}^{p-1}+\|v\|_{X_s(T)}^{p-1}\right],
			\end{align} where the positive constant $A$ is independent of $T$.  
			
			In particular,  from (\ref{final22}),  we also have 
			\begin{align}\label{final222}
				\|u^{\text{non}}\|_{X_s(T)} \leq C \|u\|_{X_s(T)}^p
			\end{align}
			and cobning it with \eqref{2number100}, we can write 
			\begin{align}\label{Final banach}
				\|N u \|_{X_s(T)}=	\| u^{\text{lin}}+u^{\text{non}}\|_{X_s(T)} \leq B\left\|\left(u_{0}, u_{1}\right)\right\|_{ \mathcal{A}^s }+B\|u\|_{X_s(T)}^{p},
			\end{align} 
			for some positive constant $B$ (independent of $T$) with initial data space $\mathcal{A}^s : =(H^s\cap \dot  H^{-\gamma}) \times (L^2\cap \dot  H^{-\gamma})$. 
			Consequently, for  some $r>1$,  we choose $R_0=rB\left\|\left(u_0, u_1\right)\right\|_{\mathcal{A}^s} $ with sufficiently small $\left\|\left(u_0, u_1\right)\right\|_{\mathcal{A}^s}<\varepsilon $  so that  
			$$BR_0^p<\frac{R_0}{r} \quad \text{and} \quad 2AR_0^{p-1}<\frac{1}{r}.$$
			Then     (\ref{final 2})   and (\ref{Final banach}) reduce to  
			\begin{align}\label{Final banach1}
				\|N u\|_{X_s(T)} \leq  \frac{2R_0}{r} 
			\end{align}  
			and \begin{align}\label{Banach11}
				\|N u-N v\|_{X_s(T)} \leq  \frac{1}{r}\|u-v\|_{X_s(T)},			\end{align}
			respectively for all $u, v\in \mathcal{B}(R_0):=\{u\in X_s(T): \|u\|_{X_s(T)}\leq R_0\}$.
			
			Since   $\left\|\left(u_0, u_1\right)\right\|_{\mathcal{A}^s}<\varepsilon $ is sufficiently small,   (\ref{Final banach1}) implies that   $Nu \in X_s(T)$, that is, $N$ maps $X_s(T)$ into itself. 
			On the other hand,   (\ref{Banach11}) implies that the map $N$ is a contraction mapping on the ball $\mathcal{B}(R_0)$ around the origin in the Banach space $X_s(T)$. Using  Banach's fixed point theorem,  we can say that there exists a uniquely determined fixed point $u^* $ of the operator $N$, which means $u^*=Nu^* \in X_s(T)$  for all positive $T$.  This implies that there exists a global-in-time small data Sobolev solution $u^*$ of the equation $ u^*=Nu^* $ in $ X_s(T)$, which also gives the solution to the semilinear damped wave equation (\ref{eq0010}).  Moreover, the fixed point satisfies the estimate $
			\|u^* \|_{X_s(T)} \leq R_0=rB\left\|\left(u_0, u_1\right)\right\|_{\mathcal{A}^s} $. 
			Due to the fact that no constant depends on  time $T$, this argument gives us the existence
			of a unique global-in-time solution in $u \in \mathcal{C}\left([0, \infty), H^s\right)$ and this completes the proof of the theorem.  
		\end{proof}

		\section{Diffusion phenomenon of damped wave equations on $\mathbb{G}$} \label{sec5}
		In this section, we establish the diffusion phenomenon with initial data belonging additionally to Sobolev spaces of negative order.   Particularly, we show how  the diffusion phenomenon bridges decay properties of solutions to the linear  Cauchy problem 
		problem
		\begin{align} \label{Linaer equation}
			\begin{cases}
				u_{tt}+\mathcal{R}u +u_{t} =0, & x\in  \mathbb{G},~t>0,\\
				u(0,x)=  u_0(x),  & x\in  \mathbb{G},\\ u_t(0, x)=  u_1(x), & x\in  \mathbb{G},
			\end{cases}
		\end{align}
		and solutions to the    Cauchy problem for the linear  heat equation  (\ref{heat}).

		Now invoking the group Fourier transform with respect to $x$ on   (\ref{heat}),  for all $\pi \in \widehat{\mathbb{G}}$,   we get   a Cauchy problem related to a parameter-dependent functional differential equation for $ \widehat{u}(t, \pi)$ as follows: 
		\begin{align*}
			\begin{cases}
				\partial_t\widehat{w}(t,\pi)+\pi(\mathcal{R}) \widehat{u}(t,\pi)  =0,& \pi \in\widehat{\G},~t>0,\\ \widehat{w}(0,\pi)=\widehat{u}_0(\pi)+\widehat{u}_1(\pi), &\pi \in\widehat{\G},
			\end{cases} 
		\end{align*}
		where  $\pi(\mathcal{R})$ is the symbol of the Rockland operator $\mathcal{R}$ on ${\mathbb{G}}$.    For $ m, k \in \mathbb{N}$,   we introduce the notation
		\begin{align}\label{basis}
			\widehat{w}(t,  \pi)_{m, k} \doteq\left(\widehat{w}(t,  \pi) e_k, e_{m}\right)_{\mathcal{H}_\pi},
		\end{align}
		where $ \{e_m\}_{m\in \mathbb{N}}$ is  the same orthonormal basis in the representation space $\mathcal{H}_\pi$ that gives   us  (\ref{matrix}). Then $\widehat{w}(t, \pi)_{m, k}$ solves the following  infinite system of  ordinary differential equation with respect to $t$    variable 
		\begin{align}\label{eqq7}
			\begin{cases}
				\partial_t\widehat{w}(t, \pi)_{m, k}+ \beta_{m, \pi}^{2 }  \widehat{w}(t, \pi)_{m, k}= 0,& \pi \in\widehat{\G},~t>0,\\ \widehat{w}(t, \pi)_{m, k}=\widehat{u}_0(\pi)_{m, k}+\widehat{u}_1(\pi)_{m, k} &\pi \in\widehat{\G}, 
			\end{cases}
		\end{align}
		where  $\beta_{m, \pi}^{2 } = \pi_m^2.$  The solution to (\ref{heat}) in the Fourier space can be written by
		\begin{align}\label{heat1}
			\widehat{w}(t, \pi)_{m, k}={e}^{-\beta_{m, \pi}^2 t}\left( \widehat{u}_0(\pi)_{m, k}+\widehat{u}_1(\pi)_{m, k}\right).
		\end{align}
		For 	$\left(u_0, u_1\right) \in\left(
		{	H} ^{s}  \cap \dot 
		{	H}^{-\gamma} \right) \times\left(
		{	H}^{s-1}  \cap  \dot
		{	H}^{-\gamma}  \right)$  with  $s \geq 0$ and $\gamma \in \mathbb{R}$ such that $s+\gamma \geq  0$,  proceeding similarly as in  the proof of   \cite[Theorem 1.1]{DKMR24},  we  have  the following 
		$ \dot
		{	H} ^{s}  $-decay estimate  for the solution to the  Cauchy problem  (\ref{heat}) as
		\begin{align}\label{heat2}
			\|w(t, \cdot)\|_{\dot
				{	H} ^{s}  } \lesssim(1+t)^{-\frac{s+\gamma}{\nu}}\left(\left\|u_0\right\|_{
				{	H}^{s}   \cap \dot
				{	H} ^{-\gamma} }+\left\|u_1\right\|_{
				{	H}^{s-1}  \cap \dot
				{	H} ^{-\gamma} }\right)
		\end{align}
		for any $t\geq 0$.

		Before going to discuss how the diffusion phenomenon bridges decay properties of solutions to the  linear damped wave equation (\ref{Linaer equation}) and solutions to the   linear  heat equation (\ref{heat}),   we recall, from \cite{DKMR24},  that the  solution to  the linear  problem (\ref{Linaer equation})  can be written in the Fourier space as
		\begin{align}\label{Asymptotic exp} 
			\widehat{u}(t, \pi)_{m, k}&= K_0(t,  \pi)_{m,k} \widehat{u}_0( \pi)_{m, k}+ K_1(t,  \pi)_{m,k}\widehat{u}_1( \pi)_{m,k},
		\end{align}
		where
		\begin{align}\label{shyam}  
			K_0(t,  \pi)_{m,k}=
			\begin{cases}
				\frac{\left(-1+\mathcal{O}(\beta_{m, \pi}^2)\right){e}^{\left(-\beta_{m, \pi}^{2 }+\mathcal{O} (\beta_{m, \pi}^{4} )\right)t}-\left(-\beta_{m, \pi}^{2 }+\mathcal{O} (\beta_{m, \pi}^{4})\right){e}^{\left(-1 +\mathcal{O} (\beta_{m, \pi}^{2} )\right)t}}{-1+\mathcal{O}\left(\beta_{m, \pi}^{2 }\right)} & \text { for }|\beta_{m, \pi}|<\delta,  \\\\ 	\frac{\left(i|\beta_{m, \pi}|-\frac{1}{2}+\mathcal{O}\left(|\beta_{m, \pi}|^{-1}\right)\right)  {e}^{\left(-i|\beta_{m, \pi}|-\frac{1}{2}+\mathcal{O}\left(|\beta_{m, \pi}|^{-1}\right)\right) t}}{2 i|\beta_{m, \pi}|+\mathcal{O}(1)} &  \\\\
				\qquad 	-\frac{\left(-i|\beta_{m, \pi}|-\frac{1}{2}+\mathcal{O}\left(|\beta_{m, \pi}|^{-1}\right)\right)  {e}^{\left(i|\beta_{m, \pi}|-\frac{1}{2}+\mathcal{O}\left(|\beta_{m, \pi}|^{-1}\right)\right) t}}{2 i|\beta_{m, \pi}|+\mathcal{O}(1)}
				& \text { for }|\beta_{m, \pi}|>N .
			\end{cases} 
		\end{align}
		and
		\begin{align} \label{shyam2}
			K_1(t,  \pi)_{m,k}
			=\left\{\begin{array}{ll}
				\frac{ {e}^{\left(-1+\mathcal{O}\left(\beta_{m, \pi}^2\right)\right) t}- {e}^{\left(-\beta_{m, \pi}^2+\mathcal{O}\left(\beta_{m, \pi}^4\right)\right) t}}{-1+\mathcal{O}\left(\beta_{m, \pi}^2\right)} & \text { for }|\beta_{m, \pi}|<\delta, \\\\
				\frac{ {e}^{\left(i|\beta_{m, \pi}|-\frac{1}{2}+\mathcal{O}\left(|\beta_{m, \pi}|^{-1}\right)\right) t}- {e}^{\left(-i|\beta_{m, \pi}|-\frac{1}{2}+\mathcal{O}\left(|\beta_{m, \pi}|^{-1}\right)\right) t}}{2 i|\beta_{m, \pi}|+\mathcal{O}(1)}  & \text { for }|\beta_{m, \pi}|>N .
			\end{array}\right.
		\end{align}
		Furthermore, from    \cite[Theorem 1.1]{DKMR24},   we have the following  	$ \dot
		{	H} ^{s}$-decay estimate for the solution  to (\ref{Linaer equation})
		\begin{align}\label{hom}
			\|u(t, \cdot)\|_{ \dot{H} ^s} \lesssim(1+t)^{-\frac{s+\gamma}{\nu}}\left(\left\|u_0\right\|_{H^s  \cap \dot{H}^{-\gamma} }+\left\|u_1\right\|_{H ^{s-1} \cap \dot{H} ^{-\gamma}}\right) ,
		\end{align}
		for any $t\geq 0$. 
		
		From  (\ref{heat2}) and (\ref{hom}), we see that the solutions to the linear damped wave equation (\ref{Linaer equation})   and for the heat equation (\ref{heat}) satisfy the same decay estimates with slightly different regularity of initial data. This observation allows us to investigate decay properties for the difference of the solutions  for the wave equation (\ref{Linaer equation})  and for the heat equation (\ref{heat}) in $\dot	{	H}_{{\Delta_{\mathbb{H}}}} ^{s} $. 
		\begin{theorem} Let ${\mathbb{G}}$ be a graded Lie group of homogeneous dimension $Q$ and let $\mathcal{R}$ be a positive Rockland operator of homogeneous degree $\nu \geq 2.$
			Let	$\left(u_0, u_1\right) \in\left(
			{	H} ^{s}  \cap \dot 
			{	H} ^{-\gamma}   \right) \times\left(
			{	H} ^{s}   \cap  \dot
			{	H} ^{-\gamma}   \right)$  with  $s \geq 0$ and $\gamma \in \mathbb{R}$ such that   $s+\gamma+\nu \geqslant 0$. Let $u$ and $w$ be the solutions to the linear Cauchy problems (\ref{Linaer equation}) and  (\ref{heat}), respectively. Then, $u-w$ satisfies
			\begin{align}\label{final}
				\|u(t, \cdot)-w(t, \cdot)\|_{
					\dot	{	H} ^{s}  } \lesssim(1+t)^{-\frac{s+\gamma}{\nu}-1}\left(\left\|u_0\right\|_{
					{	H}^{s}   \cap \dot
					{	H} ^{-\gamma}  }+\left\|u_1\right\|_{
					{	H}^{s-1}   \cap \dot
					{	H} ^{-\gamma}  }\right) .
			\end{align}
		\end{theorem} 
		\begin{proof}
			First we notice that for $|\beta_{m, \pi}|<\varepsilon \ll 1$, we see that 
			\begin{align*}
				&	\left| K_0(t,  \pi)_{m,k} - {e}^{-\beta_{m, \pi}^2 t}\right|\\
				&=\left| 	\frac{\left(-1+\mathcal{O}(\beta_{m, \pi}^2)\right){e}^{\left(-\beta_{m, \pi}^{2 }+\mathcal{O} (\beta_{m, \pi}^{4} )\right)t}-\left(-\beta_{m, \pi}^{2 }+\mathcal{O} (\beta_{m, \pi}^{4})\right){e}^{\left(-1 +\mathcal{O} (\beta_{m, \pi}^{2} )\right)t}}{-1+\mathcal{O}\left(\beta_{m, \pi}^{2 }\right)} - {e}^{-\beta_{m, \pi}^2t}\right|\\
				&	\lesssim \left| 	\frac{\left(-1+\mathcal{O}(\beta_{m, \pi}^2)\right){e}^{\left(-\beta_{m, \pi}^{2 }+\mathcal{O} (\beta_{m, \pi}^{4} )\right)t}}{-1+\mathcal{O}\left(\beta_{m, \pi}^{2 }\right)} - {e}^{-\beta_{m, \pi}^2t}\right|+e^{-ct}\\
				&	= \left| 	 {e}^{\left(-\beta_{m, \pi}^{2 }+\mathcal{O} (\beta_{m, \pi}^{4} )\right)t} - {e}^{-\beta_{m, \pi}^2t}\right|+e^{-ct}\\
			\end{align*}
			and 
			\begin{align*}
				\left| K_1(t,  \pi)_{m,k} - {e}^{-\beta_{m, \pi}^2t}\right|
				&=\left| 	\frac{ {e}^{\left(-1+\mathcal{O}\left(\beta_{m, \pi}^2\right)\right) t}- {e}^{\left(-\beta_{m, \pi}^2+\mathcal{O}\left(\beta_{m, \pi}^4\right)\right) t}}{-1+\mathcal{O}\left(\beta_{m, \pi}^2\right)} - {e}^{-\beta_{m, \pi}^2t}\right|\\
				&	\lesssim \left| 	\frac{-{e}^{\left(-\beta_{m, \pi}^{2 }+\mathcal{O} (\beta_{m, \pi}^{4} )\right)t}}{-1+\mathcal{O}\left(\beta_{m, \pi}^{2 }\right)} - {e}^{-\beta_{m, \pi}^2t}\right|+e^{-ct}
			\end{align*}
			for some positive constant $c$. 	Similarly, 	 for $|\beta_{m, \pi}|	\gg   1$, we see that 
			\begin{align*}
				&	\left| K_0(t,  \pi)_{m,k} - {e}^{-\beta_{m, \pi}^2 t}\right|\\
				&\leq \left| 		\frac{\left(i|\beta_{m, \pi}|-\frac{1}{2}+\mathcal{O}\left(|\beta_{m, \pi}|^{-1}\right)\right)  {e}^{\left(-i|\beta_{m, \pi}|-\frac{1}{2}+\mathcal{O}\left(|\beta_{m, \pi}|^{-1}\right)\right) t}}{2 i|\beta_{m, \pi}|+\mathcal{O}(1)} - {e}^{-\beta_{m, \pi}^2t}\right|\\& \quad +\left| 	\frac{\left(-i|\beta_{k \lambda}|-\frac{1}{2}+\mathcal{O}\left(|\beta_{m, \pi}|^{-1}\right)\right)  {e}^{\left(i|\beta_{m, \pi}|-\frac{1}{2}+\mathcal{O}\left(|\beta_{m, \pi}|^{-1}\right)\right) t}}{2 i|\beta_{m, \pi}|+\mathcal{O}(1)}- {e}^{-\beta_{m, \pi}^2t}\right|\\
				&\leq  	  {e}^{\left( -\frac{1}{2}+\mathcal{O}\left(|\beta_{m, \pi}|^{-1}\right)\right) t}  + {e}^{- t} + {e}^{\left(-\frac{1}{2}+\mathcal{O}\left(|\beta_{m, \pi}|^{-1}\right)\right) t} + {e}^{- t} \\&\lesssim e^{-ct},
			\end{align*}
			and 
			\begin{align*}
				\left| K_1(t,  \pi)_{m,k} - {e}^{-\beta_{m, \pi}^2t}\right|
				&=\left| 	\frac{ {e}^{\left(i|\beta_{m, \pi}|-\frac{1}{2}+\mathcal{O}\left(|\beta_{m, \pi}|^{-1}\right)\right) t}- {e}^{\left(-i|\beta_{m, \pi}|-\frac{1}{2}+\mathcal{O}\left(|\beta_{m, \pi}|^{-1}\right)\right) t}}{2 i|\beta_{m, \pi}|+\mathcal{O}(1)}  - {e}^{-\beta_{m, \pi}^2t}\right|\\
				&\leq \frac{ {e}^{\left(-\frac{1}{2}+\mathcal{O}\left(|\beta_{m, \pi}|^{-1}\right)\right) t}+ {e}^{\left(-\frac{1}{2}+\mathcal{O}\left(|\beta_{m, \pi}|^{-1}\right)\right) t}}{\sqrt{4\beta_{m, \pi}^2+\mathcal{O}(1)^2}}  + {e}^{-t}\\
				&\lesssim  (|\beta_{m, \pi}|+1)^{-1}   {e}^{-ct}
			\end{align*}
			for some positive constant $c$.
			
			Now define the cut-off function $\chi_{\text {int }}$ with    its support in $\{ 	\beta_{m, \pi} \in \mathbb{R}^*: |\beta_{m, \pi}|<\varepsilon \ll 1\}$. Then using the above estimates,  directly we have 
			\begin{align}\label{heat111}\nonumber
				&\left|\widehat{u}(t,\lambda)_{kl}-	\widehat{w}(t,\lambda)_{kl}\right| \\\nonumber
				&= | \left( K_0(t,  \pi)_{m,k}-{e}^{-\beta_{m, \pi}^2t}\right)  \widehat{u}_0(\pi)_{m, k}+\left( K_1(t,  \pi)_{m,k}-{e}^{-\beta_{m, \pi}^2 t}\right) \widehat{u}_1(\pi)_{m, k} |\\	\nonumber
				&	= \chi_{\text {int }} | \left( K_0(t,  \pi)_{m,k}-{e}^{-\beta_{m, \pi}^2t}\right)  \widehat{u}_0(\pi)_{m, k}+\left( K_1(t,  \pi)_{m,k}-{e}^{-\beta_{m, \pi}^2t}\right) \widehat{u}_1(\pi)_{m, k} |\\\nonumber
				&\qquad +(1-\chi_{\text {int }}) | \left( K_0(t,  \pi)_{m,k}-{e}^{-\beta_{m, \pi}^2t}\right)  \widehat{u}_0(\pi)_{m, k}+\left( K_1(t,  \pi)_{m,k}-{e}^{-\beta_{m, \pi}^2t}\right) \widehat{u}_1(\pi)_{m, k} |\\\nonumber
				& 		\lesssim \chi_{\text {int }} \left( \left| 	 {e}^{\left(-\beta_{m, \pi}^{2 }+\mathcal{O} (\beta_{m, \pi}^{4} )\right)t} - {e}^{-\beta_{m, \pi}^2t}\right|+e^{-ct} \right)\left| \widehat{u}_0(\pi)_{m, k}\right| \\\nonumber
				& \qquad + \chi_{\text {int }} \left( \left| 	\frac{-{e}^{\left(-\beta_{m, \pi}^{2 }+\mathcal{O} (\beta_{m, \pi}^{4} )\right)t}}{-1+\mathcal{O}\left(\beta_{m, \pi}^{2 }\right)} - {e}^{-\beta_{m, \pi}^2t}\right|+e^{-ct}				
				\right)\left| \widehat{u}_1(\pi)_{m, k}\right| \\
				\nonumber
				& \qquad+\left(1-\chi_{\text {int}}\right)  {e}^{-c t}\left| \widehat{u}_0(\pi)_{m, k}\right| \\&\qquad +\left(1-\chi_{\text {int }}(\xi)\right) (|\beta_{m, \pi}|+1)^{-1}   {e}^{-ct} \left| \widehat{u}_1(\pi)_{m, k}\right| .
			\end{align}		
			For $|\beta_{m, \pi}|<\varepsilon \ll 1$, using the Newton-Leibniz integral formula,  we have
			$$
			\begin{aligned}
				{e}^{\left(-\beta_{m, \pi}^{2 }+\mathcal{O} (\beta_{m, \pi}^{4} )\right)t} - {e}^{-\beta_{m, \pi}^2t} & = {e}^{-\beta_{m, \pi}^2t}  \mathcal{O} (\beta_{m, \pi}^{4})   \int_0^t   e^{\mathcal{O} (\beta_{m, \pi}^{4} ) s } ~ds\\&
				= {e}^{-\beta_{m, \pi}^2t}  \mathcal{O} (\beta_{m, \pi}^{4})   \int_0^t   e^{\mathcal{O} (\beta_{m, \pi}^{2} ) s } ~ds\\
				& \leq C  \beta_{m, \pi}^{2}  {e}^{-c\beta_{m, \pi}^2t},
			\end{aligned}
			$$ where $C$ is a positive constant. 			Thus from (\ref{heat111}), we  obtain  
			\begin{align}\label{coefficient} \nonumber
				\left|\widehat{u}(t,\lambda)_{kl}-	\widehat{w}(t,\lambda)_{kl}\right| 
				& 		\lesssim \chi_{\text {int }} \left(  \beta_{m, \pi}^{2}  {e}^{-c\beta_{m, \pi}^2t}+e^{-ct} \right)\left( \left| \widehat{u}_0(\pi)_{m, k}\right|+\left| \widehat{u}_1(\pi)_{m, k}\right| \right)   \\&\qquad +\left(1-\chi_{\text {int }}\right)\left( \left| \widehat{u}_0(\pi)_{m, k}\right|+\left| \widehat{u}_1(\pi)_{m, k}\right| \right).
			\end{align}
			Now following same procedure  as in  the proof of  in  \cite[Theorem 1.1]{DKMR24},  we can conclude that 
			$$
			\|u(t, \cdot)-w(t, \cdot)\|_{
				\dot	{	H} ^{s} } \lesssim(1+t)^{-\frac{s+\gamma}{\nu}-1}\left(\left\|u_0\right\|_{
				{	H}^{s}  \cap \dot
				{	H} ^{-\gamma} }+\left\|u_1\right\|_{
				{	H}^{s-1} \cap \dot
				{	H} ^{-\gamma}  }\right),
			$$ completing the proof.  \end{proof}
		We may immediately make the following observations based on the preceding finding.   \begin{rem}
			Based on the results obtained from (\ref{final}), it is evident that the decay rate is significantly higher compared to (\ref{hom}) and (\ref{heat2}).   This is due to the    additional coefficient  $\beta_{m, \pi}^{2} $  with the factor $  {e}^{-c\beta_{m, \pi}^2t}$ in  (\ref{coefficient}). 
		\end{rem}
		\begin{rem}
			Concerning the decay rate in (\ref{final}), we see that the decay is enhanced by a factor of $(1+t)^{-1}$ when we subtract the solution to wave equation (\ref{Linaer equation}) by the solution to heat equation (\ref{heat}).  This concludes that the diffusion phenomenon is also valid in the framework of the negative order Sobolev  space $\dot{H}^{-\gamma}$. 
		\end{rem}

		\section*{Acknowledgement}
		VK, MR and BT are supported by the FWO Odysseus 1 grant G.0H94.18N: Analysis and Partial Differential Equations, the Methusalem program of the Ghent University Special Research Fund (BOF) (Grant number 01M01021). VK and MR are also supported by FWO Senior Research Grant G011522N. MR is also supported by EPSRC grant EP/V005529/1.  SSM is supported by the DST-INSPIRE Faculty Fellowship DST/INSPIRE/04/2023/002038. BT is also supported by the Science Committee of the Ministry of Education and Science of the Republic of Kazakhstan (Grant No. AP14869090).
		
		%		\section{Data availability statement}
		%		The authors confirm that the data supporting the findings of this study are available within the article  and its supplementary materials.
		%		
		%		
		%		\section{Declarations}\vspace{0.1cm}
		%	 \textbf{Ethical Approval:} Not applicable.\vspace{0.1cm}
		%		
		%		
		%		\textbf{Competing interests:} No potential competing of interest was reported by the author.\vspace{0.1cm}
		%		
		%		
		%		

		%	\textbf{Availability of data and materials:} All the data uned are  within the manuscript.


\begin{thebibliography}{99}
			
			
			
			
		%	\baselineskip=11pt
			\baselineskip=17pt
			
			
			\bibitem{BF89} H. Bellout and A. Friedman,  Blow-up estimates for a nonlinear hyperbolic heat equation, {\it SIAM J. Math. Anal.}, 20, 354–366 (1989). 
			
			
			\bibitem{Reissig} 	W.  Chen and M.  Reissig,   On the critical exponent and sharp lifespan estimates for semilinear damped wave equations with data from Sobolev spaces of negative order, \emph{J. Evol. Equ.} 23, 13 (2023).  
			
			
			
			
			
			\bibitem{DKMR}  A. Dasgupta, V. Kumar, S. S. Mondal  and M. Ruzhansky, Semilinear damped wave equations on the Heisenberg group with initial data from Sobolev spaces of negative order,  {\it  J. Evol. Equ.}    24, 51,  (2024). 
			
			
			\bibitem{DKMR24}  A. Dasgupta, V. Kumar, S. S. Mondal,  and M. Ruzhansky, Higher order hypoelliptic damped wave equations on graded Lie groups with data from negative order Sobolev spaces,   arxiv preprint (2024). arXiv.2404.08766 
			
			
			\bibitem{Dcri} M. D'Abbicco, Semilinear damped wave equations with data from Sobolev spaces of negative order: the critical case in Euclidean setting and in the Heisenberg space, arxiv preprint  (2024). arXiv.2408.11756
			
			
			\bibitem{Fischer}	V. Fischer and  M. Ruzhansky, \emph{Quantization on Nilpotent Lie Groups}, Progr. Math., vol.314, Birkh\"auser/Springer, (2016).
			
			
			\bibitem{RF17}  V. Fischer and M. Ruzhansky,  Sobolev spaces on graded Lie groups,  \emph{Ann. Inst. Fourier (Grenoble)} 67(4), 1671–1723 (2017).
			\bibitem{Folland} 	G.B. Folland and E.M. Stein, Hardy spaces on homogeneous groups, Princeton University	Press, Princeton, New Jersey (1982). 	
			\bibitem{Fujita66} H. Fujita, On the blowing up of solutions of the Cauchy problem for $u_t= \Delta u+u^{1+\alpha}$, {\it J. Fac. Sci. Univ. Tokyo Sect. I} (13), 109–124 (1966)
			
			\bibitem{RG98} T. Gallay and G. Raugel,  Scaling variables and asymptotic expansions in damped wave equations, {\it J. Differential Equations} 150, 42–97 (1998)
			
			\bibitem{Vla}	 V. Georgiev and A. Palmieri, Critical exponent of Fujita-type for the semilinear damped wave equation on the Heisenberg group with power nonlinearity, \emph{J. Differential Equations} 269(1), 420-448 (2020).
			
			\bibitem{palmieri} 	 V.  Georgiev and  A. Palmieri,	Lifespan estimates for local in time solutions	to the semilinear heat equation on the Heisenberg group, \emph{Ann. Mat. Pura Appl.} 200, 999-1032 (2021).
			
			\bibitem{GW} Y.  Guo and Y.  Wang, Decay of dissipative equations and negative Sobolev spaces, \emph{Comm. Partial Differential Equations} 37(12), 2165-2208 (2012). 
			
			\bibitem{HKN04} N. Hayashi, E. I. Kaikina and P. I. Naumkin, Damped wave equation with super critical nonlinearities,{ \it Differential Integral Equations} 17,  637–652 (2004).
			\bibitem{HN} B. Helffer and J, Nourrigat, Caracterisation des operateurs hypoelliptiques homogenes invariants a gauche sur un groupe de Lie nilpotent gradue. {\it Comm. Partial Differential Equations} 4(8), 899–958, (1979).
			
			\bibitem{spectrum} A. Hulanicki, J. W. Jenkins, and J. Ludwig, Minimum eigenvalues for positive, Rockland operators, \emph{Proc. Amer. Math. Soc.} 94,  718–720 (1985).	
			
			\bibitem{HL92} L. Hsiao and T. P. Liu, Convergence to nonlinear diffusion waves for solutions of a system of
			hyperbolic conservation laws with damping, \emph{Comm. Math. Phys.}, 143(3), 599–605 (1992).
			
			\bibitem{Ikeda2019} M. Ikeda, T. Inui, M. Okamoto and  Y. Wakasugi, $L^p-L^q$ estimates for the damped wave equation and the critical exponent for the nonlinear problem with slowly decaying data, \emph{Commun. Pure Appl. Anal.} 18(4), 1967-2008 (2019).
			
			\bibitem{Ikeda2002} R. Ikehata and M. Ohta, Critical exponents for semilinear dissipative wave equations in $\mathbb{R}^N$, \emph{J. Math. Anal. Appl.} 269(1), 87-97 (2002).	
			\bibitem{IKeta and Tanizawa}  R. Ikehata and K. Tanizawa, Global existence of solutions for semilinear damped wave equations in $\mathbb{R}^n$ with noncompactly supported initial data, \emph{Nonlinear Anal.} 61(7), 1189-1208  (2005).
			
			\bibitem{JKS15} M. Jleli, M. Kirane, B. Samet, Nonexistence results for a class of evolution equations in the Heisenberg group, {\it Fract. Calc. Appl. Anal.} 18(3), 717-734 (2015).
			
			\bibitem{KTT20} A. Kassymov, N. Tokmagambetov, and B. Torebek, Nonexistence results for the hyperbolic-type equations on graded Lie groups, {\it Bull. Malays. Math. Sci. Soc.} 43, 4223-4243 (2020).
			
			\bibitem{KR15} M. Kirane, L. Ragoub, Nonexistence results for a pseudo-hyperbolic equation in the Heisenberg group, {\it Electron. J. Differential Equations}, 2015, 110, (2015).
			
			
			\bibitem{Kar00} G. Karch, Selfsimilar profiles in large time asymptotics of solutions to damped wave equations,
			\emph{Studia Math.}, 143(2), 175–197 (2000).
			
			\bibitem{Lin}  J. Lin, K. Nishihara, and  J. Zhai, Critical exponent for the semilinear wave equation with
			time-dependent damping, \emph{Discrete Contin. Dyn. Syst.}, 32, 4307-4320 (2012).
			
			\bibitem{LN} T.Y. Lee and W.M. Ni, Global existence, large time behavior and life span on solution of a
			semilinear parabolic Cauchy problem, {\it Trans. Amer. Math. Soc.} 333, 365–378 (1992).
			\bibitem{Matsumura}  A. Matsumura, On the asymptotic behavior of solutions of semi-linear wave equations, \emph{Publ. Res. Inst. Math. Sci.} 12(1), 169-189 (1976/77).
			\bibitem{Miller} K. G. Miller, Parametrices for hypoelliptic operators on step two nilpotent Lie groups, {\it Comm. Partial Differential
				Equations}, 5(11), 1153–1184, (1980).
			
			\bibitem{MS99} D. Müller and E. M. Stein,  $L^p$-estimates for the wave equation on the Heisenberg group, {\it Revista Matemática Iberoamericana}, 15(2), 297-332,  (1999).	
			
			
			\bibitem{24} A. I. Nachman, The wave equation on the Heisenberg group, \emph{Comm. Partial Differential Equations} 7(6),   675-714 (1982).
			
			
			\bibitem{Nakao93}  M. Nakao and  K. Ono, Existence of global solutions to the Cauchy problem for the semilinear dissipative wave equations, \emph{Math. Z.} 214(2), 325-342 (1993).
			
			\bibitem{Palmieri2020}	A. Palmieri, Decay estimates for the linear damped wave equation on the
			Heisenberg group, \emph{J. Funct. Anal.}, 279(9), 108721 (2020).
			
			\bibitem{Nis03} K. Nishihara, $L^p-L ^q$ estimates of solutions to the damped wave equation in 3-dimensional space
			and their application, \emph{Math. Z.}, 244(3), 631–649 (2003).
			
			\bibitem{Pas98} A. Pascucci, Semilinear equations on nilpotent Lie groups: global existence and blow-up of solutions, {\it  Matematiche}, 53(2), 345–357 (1998). 
			
			\bibitem{Poho}	 S. I. Pohozaev and  L. V\'eron,  Non existence results of solutions of semilinear differential inequalities  on the Heisenberg group, \emph{Manuscripta Math.} 102,85-99 (2000).
			\bibitem{Rock} C. Rockland, Hypoellipticity on the Heisenberg group-representation-theoretic criteria. \emph{Trans. Amer. Math. Soc.},
			240, 1–52, (1978).	
			
			\bibitem{gra1}   M. Ruzhansky and C. Taranto, Time-dependent wave equations on graded groups, \emph{Acta Appl. Math.} 171, Article number: 21 (2021).			
			
			\bibitem{30} M. Ruzhansky and N. Tokmagambetov, Nonlinear damped wave equations for the sub-Laplacian on the Heisenberg group and for Rockland operators on graded Lie groups, \emph{J. Differential Equations} 265(10), 5212-5236 (2018).
			
			\bibitem{RY} M. Ruzhansky and N. Yessirkegenov, Existence and non-existence of global solutions for semilinear heat equations and inequalities on sub-Riemannian manifolds, and Fujita exponent on unimodular Lie groups, \emph{J. Differential Equations} 308, 455-473 (2022).
			
			\bibitem{gra3}   C. Taranto, \emph{Wave equations on graded groups and hypoelliptic Gevrey spaces}, Imperial College London Ph.D. thesis, (2018). 
			
			\bibitem{terRob97} A. F. M. ter Elst and Derek W. Robinson,  Spectral estimates for positive Rockland operators. Algebraic groups and Lie groups, 195–213, Austral. Math. Soc. Lect. Ser., 9, Cambridge Univ. Press, Cambridge, 1997	
			
			\bibitem{TZZ24} H. Tang,  S. Zhang and W. Zou, Decay of the compressible Navier-Stokes equations with hyperbolic heat conduction, {\em J. Differential Equations} 388, 1-33 (2024).
			
			\bibitem{thanga} 	S. Thangavelu,  \emph{Harmonic analysis on the Heisenberg group}, volume 159, Progress in Mathematics, Springer (1998).
			
			\bibitem{Todorova}  G. Todorova and  B. Yordanov, Critical exponent for a nonlinear wave equation with damping, \emph{J. Differential Equations} 174(2), 464-489 (2001).
			
			\bibitem{Yang} Z. Yang,  Fujita exponent and nonexistence result for the Rockland heat equation, \emph{Appl. Math. Lett.} 121, 107386 (2021).
			
			\bibitem{Zhang1}
			Q. S. Zhang, The critical exponent of a reaction diffusion equation on some Lie groups, {\it Math. Z.} 228, 51--72 (1998).			
			
			\bibitem{Zhang} 	 Q. S. Zhang, A blow-up result for a nonlinear wave equation with damping: the critical case, \emph{C. R. Acad. Sci. Paris Sér. I Math.} 333(2), 109-114  (2001). 
		\end{thebibliography}
	\end{document}